\documentclass[a4paper,11pt,leqno]{article}

\usepackage[latin1]{inputenc}
\usepackage[T1]{fontenc} %only using latex
\usepackage[english]{babel}
\usepackage{verbatim}
\usepackage{amsfonts}
\usepackage{amsmath}
\usepackage{amsthm}
\usepackage{amssymb}
\usepackage{color}
\usepackage{graphicx}
\usepackage{bm}

\theoremstyle{definition}
\newtheorem{defin}{Definition}[section]

\newtheorem{rem}[defin]{Remark}

\theoremstyle{plane}
\newtheorem{thm}[defin]{Theorem}
\newtheorem{prop}[defin]{Proposition}
\newtheorem{coroll}[defin]{Corollary}
\newtheorem{lemma}[defin]{Lemma}

\newcommand{\mbb}{\mathbb}

\newcommand{\mc}{\mathcal}
\newcommand{\veps}{\varepsilon}
\newcommand{\what}{\widehat}
\newcommand{\wtilde}{\widetilde}
\newcommand{\vphi}{\varphi}
\newcommand{\oline}{\overline}

\newcommand{\ra}{\rightarrow}

\newcommand{\hra}{\hookrightarrow}

\newcommand{\g}{\gamma}

\newcommand{\R}{\mathbb{R}}
\newcommand{\C}{\mathbb{C}}
\newcommand{\N}{\mathbb{N}}
\newcommand{\Z}{\mathbb{Z}}

\renewcommand{\Re}{{\rm Re}\,}

\newcommand{\Id}{{\rm Id}\,}

\allowdisplaybreaks

\def\d{\partial}

\title{\large{\bfseries{\textsc{The well-posedness issue in Sobolev spaces for \\ hyperbolic systems
with Zygmund-type coefficients}}}}

\author{\textsl{Ferruccio Colombini}$\,^1\;$, \textsl{Daniele Del Santo}$\,^2\;$, \textsl{Francesco Fanelli}$\,^3\;$,
\textsl{Guy M\'etivier}$\,^{4}$ \vspace{.5cm} \\
\small{$\,^1\;$ \textsc{Universit\`a di Pisa}} \\ {\small Dipartimento di Matematica} \\
\small{\ttfamily{colombini@dm.unipi.it}} \vspace{0.3cm} \\
\small{$\,^2\;$ \textsc{Universit\`a di Trieste}} \\ {\small Dipartimento di Matematica e Geoscienze} \\
\small{\ttfamily{delsanto@units.it}} \vspace{0.3cm} \\
\small{$\,^3\;$ \textsc{Universit\'e Paris-Diderot -- Paris 7}} \\
{\small Institut de Math\'ematiques de Jussieu-Paris Rive Gauche UMR 7586} \\
\small{\ttfamily{francesco.fanelli@imj-prg.fr}} \vspace{0.3cm} \\
\small{$\,^4\;$ \textsc{Universit\'e de Bordeaux 1}} \\ {\small Institut de Math\'ematiques de Bordeaux UMR 5251} \\
\small{\ttfamily{guy.metivier@math.u-bordeaux1.fr}}}

\vspace{.2cm}
\date\today

\begin{document}

\maketitle

\subsubsection*{Abstract}
{\small In this paper we study the well-posedness of the Cauchy problem for first order
hyperbolic systems with constant multiplicities and with low regularity coefficients depending just on the time variable.
We consider Zygmund and log-Zygmund type assumptions, and we prove well-posedness in $H^\infty$ respectively
without loss and with finite loss of derivatives. The key to obtain the results is the construction of a suitable symmetrizer for our
system, which allows us to recover energy estimates (with or without loss) for the hyperbolic operator
under consideration. This can be achievied, in contrast with the classical case of systems with smooth (say Lipschitz) coefficients,
by adding one step in the diagonalization process, and building the symmetrizer up to the second order.}

\paragraph*{Mathematical Subject Classification (2010):}{\small 35L45 (primary); 35B45, 35B65 (secondary).}

\paragraph*{Keywords:}{\small hyperbolic system with constant multiplicities, Zygmund and log-Zygmund conditions,
microlocal symmetrizability, energy estimates, $H^\infty$ well-posedness.}

%%%%%%%%%%%%%%%%%%%%%%%%%%%%%%%%%%%%%%%%%%%%%%%%%%%%%%%%%
%%%%%%%%%%%%%%%%%%%%%%%%%%%%%%%%%%%%%%%%%%%%%%%%%%%%%%%%%
\section{Introduction}
%%%%%%%%%%%%%%%%%%%%%%%%%%%%%%%%%%%%%%%%%%%%%%%%%%%%%%%%%
%%%%%%%%%%%%%%%%%%%%%%%%%%%%%%%%%%%%%%%%%%%%%%%%%%%%%%%%%

The present paper is devoted to the analysis of the well-posedness of the Cauchy problem related to a first order hyperbolic system,
\begin{equation} \label{intro_eq:L}
 Pu(t,x)\,:=\,\d_tu(t,x)\,+\,\sum_{j=1}^n A_j(t,x)\,\d_ju(t,x)\,,
\end{equation}
under low regularity assumptions on its coefficients.
Here, $(t,x)\in[0,T]\times\R^n$ for some fixed time $T>0$ and integer $n\geq1$,
the vectors $u(t,x)$ and $Pu(t,x)$ belong to $\R^m$ for some $m\geq1$, and the $A_j(t,x)$'s are $m\times m$ real-valued matrices, which
will be assumed to be non-Lipschitz with respect to the time variable.

\medbreak
This study represents a natural extension of the investigation about the well-posedness in Sobolev classes for the scalar wave operator
\begin{equation} \label{intro_eq:wave}
 Wu(t,x)\,:=\,\d_t^2u(t,x)\,-\,\sum_{j,k=1}^n\d_j\bigl(a_{jk}(t,x)\,\d_ku(t,x)\bigr)\,,
\end{equation}
under the hypothesis of symmetry, boundedness and strict hyperbolicity.
Namely, this means that $a_{jk}=a_{kj}$ for any $1\leq j,k\leq n$, and that there exist two positive constants
$0<\lambda_0\leq\Lambda_0$ such that, for any $(t,x)\in[0,T]\times\R^n$ and any $\xi\in\R^n$,
$$
\lambda_0\,|\xi|^2\,\leq\,\sum_{j,k=1}^n a_{jk}(t,x)\,\xi_j\,\xi_k\,\leq\,\Lambda_0\,|\xi|^2\,.
$$

In \cite{H-S}, Hurd and Sattinger proved that, if the coefficients  of operator $W$ are Lipschitz continuous in time, and even just
bounded with respect to $x$, then the related Cauchy problem is well-posed in the energy space $H^1\times L^2$. The result extends to
higher regularity Sobolev spaces if more smoothness in space variables is assumed for the $a_{jk}$'s.

After \cite{H-S}, a large number of works have been devoted to recover well-posedness for operators with non-Lipschitz coefficients,
possibly compensating this lack of smoothness with suitable hypotheses with respect to $x$.

In work \cite{C-DG-S}, Colombini, De Giorgi and Spagnolo considered the case $a_{jk}=a_{jk}(t)$, and they introduced an
integral log-Lipschitz assumption: there exists a constant $C>0$ such that, for any $\veps\in\,]0,T]$, one has
\begin{equation} \label{intro_hyp:int-LL}
\int^{T-\veps}_0\left|a_{jk}(t+\veps)\,-\,a_{jk}(t)\right|dt\;\leq\;C\,\veps\,\log\left(1\,+\,\frac{1}{\veps}\right)\,.
\end{equation}
Under this condition, they were able to prove an energy estimate for $W$ with a \emph{fixed loss of derivatives}:
there exists a constant $\delta>0$ such that, for all $s\in\R$, the inequality
\begin{eqnarray}
  &&\sup_{0\le t \le T} \biggl(\|u(t)\|_{H^{s+1-\delta}}\,  +
 \|\partial_t u(t)\|_{H^{s-\delta}}\biggr)\,\leq \label{intro_est:c-loss} \\
&&\qquad\qquad\qquad\qquad\qquad\qquad
\leq\, C_s \left(\|u(0)\|_{H^{s+1}}+
 \|\partial_t u(0)\|_{H^s} + \int_0^{T}  \|  W u(t)\|_{H^{s-\delta}}\, dt\right) \nonumber
\end{eqnarray}
holds true for all $u\in\mc{C}^2([0,T];H^\infty(\R^n))$, for some constant $C_s$ depending only on $s$. In particular,
such an estimate implies the well-posedness of the Cauchy problem for $W$, but just in the space $H^\infty$, due to
the loss of regularity involving $u$.

Let us immediately point out that, under a (stronger) pointwise log-Lipschitz condition, instead,
one can get a loss of derivatives which increases in time: $\delta$ in estimate \eqref{intro_est:c-loss} is replaced by
$\beta\,t$, for some constant $\beta>0$ depending just on the coefficients of $W$.

Moreover, in \cite{Cic-C} Cicognani and Colombini proved, by construction of explicit counterexamples, a classification
of the relation between modulus of continuity of the $a_{jk}$'s and loss of derivatives in the energy estimates. This result in particular
implies the sharpness of estimate \eqref{intro_est:c-loss} (actually, of its time-dependent version under the pointwise hypothesis),
when one just looks at the modulus of continuity of the coefficients.

The recent work \cite{Tar} by Tarama changed the point of view where looking at this problem from. There, again in the case
$a_{jk}=a_{jk}(t)$, the author introduced
conditions on the second variation of the coefficients, rather than on their modulus of continuity (i.e. the first variation).
In particular, he considered integral Zygmund and log-Zygmund conditions, which read (with $\ell=0$ for the former, $\ell=1$ for the latter):
\begin{equation} \label{intro_hyp:int-LZ}
 \int^{T-\veps}_\veps\left|a_{jk}(t+\veps)\,+\,a_{jk}(t-\veps)\,-\,2\, a_{jk}(t)\right|dt\;\leq\;
C\,\veps\,\log^\ell\left(1\,+\,\frac{1}{\veps}\right)\,,
\end{equation}
for some constant $C>0$ and for all $\veps\in\,]0,T/2[\,$. Note that these assumptions are weaker than the respective ones involving
the first difference of the $a_{jk}$'s; on the other hand, they are related (for regular functions) to the second derivative, so they
set in a different context (in particular, the result of \cite{Cic-C} doesn't apply).
Tarama proved an energy estimate with \emph{no loss of derivatives} in any Sobolev spaces $H^s\times H^{s-1}$ in the Zygmund instance,
and an estimate with fixed loss, analogous to \eqref{intro_est:c-loss}, in the log-Zygmund one.

\medbreak
Let us come back to the case of the hyperbolic system $P$ defined by \eqref{intro_eq:L}
and let us assume that $P$ is \emph{strictly hyperbolic} or \emph{hyperbolic with constant multiplicities}:
for the correct definitions, we refer to Section \ref{s:results}. For the time being, it's enough to keep in mind that both these assumptions
imply that the $m\times m$ matrix
$$
\mc{A}(t,x,\xi)\,:=\,\sum_{j=1}^n \xi_j\,A_j(t,x)
$$
is diagonalizable at any point $(t,x,\xi)\in[0,T]\times\R^n\times\R^n$.

As said before, we want to extend the analysis performed for the wave operator $W$.
As a matter of fact, the context is analogous to the one for second order hyperbolic equations: if the
coefficients are Lipschitz continuous both in time and space variables, then the Cauchy problem for $P$ is well-posed
in the energy space $L^2$. For such a result, one can refer to \cite{E}, Chapter 7, or to \cite{M-2008}, Chapter 7, and the references
therein. In the former textbook, the result is proved (actually for symmetric systems) by use of a vanishing viscosity argument;
in the latter, instead, it is proved resorting to techniques coming from paradifferential calculus, in view of applications
to the non-linear case (developed in the following chapter).

As for symmetric systems, it's worth mentioning paper \cite{B-C_1994} by Bahouri and Chemin: there the authors were able to prove that
the phenomenon of loss of derivatives in the energy estimates, that we mentioned above for the wave operator $W$,
occurs also for the special case of transport equations with $L^1_t(LL_x)$ coefficients (where we denoted with
$LL$ the class of log-Lipschitz functions). Moreover, they applied their result to the study of the homogeneous incompressible
Euler equations; we refer to work \cite{D-P} by Danchin and Paicu for another application, in the context of the Boussinesq system.
We also refer to paper \cite{D_2005} by Danchin for further estimates (with and without loss
of derivatives) for transport and transport-diffusion equations in Besov spaces.

This having been said, let us explain the main ideas of the arguments used in \cite{M-2008} (see also Chapters 2 and 3 of the same book)
for the study of general hyperbolic systems with constant multiplicities under Lipschitz regularity hypothesis.
The main issue is to construct a scalar product with respect to which
the matrix symbol $\mc{A}$ is self-adjoint: this can be done by use of the projection operators over the eigenspaces
related to $\mc{A}$.
In fact, hyperbolicity with constant multiplicities implies microlocal symmetrizability,
in the sense of M\'etivier (see \cite{M-2008}, Chapter 7).
\begin{defin} \label{d:micro_symm}
 System \eqref{intro_eq:L} is \emph{microlocal symmetrizable} if there exists a $m\times m$ matrix $S(t,x,\xi)$,
 homogeneous of degree $0$ in $\xi$, such that:
 \begin{itemize}
  \item $\xi\,\mapsto\,S(t,x,\xi)$ is $\mc{C}^\infty$ for $\xi\neq0$;
  \item $(t,x)\,\mapsto\,S(t,x,\xi)$ is $W^{1,\infty}$ for $(t,x)\in[0,T]\times\R^n$;
  \item for any point $(t,x,\xi)$, the matrix $S(t,x,\xi)$ is self-adjoint;
  \item there exists a constant $\lambda>0$ such that $S(t,x,\xi)\,\geq\,\lambda\,\Id$ for any $(t,x,\xi)$;
  \item for any point $(t,x,\xi)$, the matrix $S(t,x,\xi)\,\mc{A}(t,x,\xi)$ is self-adjoint.
 \end{itemize}
The matrix valued function $S$ is called a \emph{symmetrizer} for system \eqref{intro_eq:L}.
\end{defin}

Once a symmetrizer for $P$ is found, then, roughly (the dependence on $x$ of the $A_j$'s makes things technically more complicated),
one can define the energy as the $L^2$ norm of $u$ with respect to this scalar product. Thanks to Lipschitz
regularity assumptions, differentiation in time and Gronwall's lemma easily allow one to find energy estimates.

\medbreak
In the present paper, we will deal with just time-dependent matrices $A_j$,
under the condition of hyperbolicity with constant multiplicities.
We will assume Zygmund and log-Zygmund type hypothesis, in the same spirit of those considered by Tarama in
paper \cite{Tar}, even if, in general, the pure integral one, i.e. \eqref{intro_hyp:int-LZ}, will be precluded to us.

In the Zygmund instance, we will prove energy estimates with no loss of derivatives for $P$, which imply the well-posedness
of the Cauchy problem in $H^s$ for any $s\in\R$. In the log-Zygmund case, instead,
we will show energy estimates with time-depedent loss of derivatives, which are still suitable to recover the well-posedness
of the Cauchy problem, but only in the space $H^\infty$.

In order to prove the results, we will combine the technique for systems we have just explained, with
ideas coming from the analysis of second order equations \eqref{intro_eq:wave}, which actually go back to paper \cite{C-DG-S}.
So, first of all we regularize the coefficients by convolution with a smoothing kernel. Then,
we pass to the phase space by Fourier transform and, at any point $(t,\xi)$, we construct a symmetrizer for the
approximated system, which will be smooth in time. Nevertheless, due to the low regularity assumptions
on the coefficients, we need to introduce a \emph{second step} in the diagonalization process, and add a lower order term to our symmetrizer.
On the one hand, thanks to the second term, we can compensate the bad behaviour of the time derivative (in the energy estimates) of the
principal part of the symmetrizer. On the other hand, linking the approximation parameter with the dual variable
(following the original idea of \cite{C-DG-S}) we are able to control the loss coming from the time derivative of the second
part (which, we recall, is of lower order).

Let us spend a few words on how constructing the symmetrizer. As we mentioned before, in the smooth case, if we denote by $\Pi_j$
the projection operators over the eigenspace $E_j$, the symmetrizer $S$ is defined by $S\,:=\,\sum_j\Pi_j^*\,\Pi_j$.
By analogy, in our context the principal part $S^0$ of the symmetrizer will be defined by a similar formula,
but we will need to introduce a ``suitable'' self-adjoint operator $\Sigma^0_j$ acting on the corresponding eigenspace.
Then, we will add a lower order term $S^1$: for each $j$, it will introduce ``suitable'' corrections in the eigenspace $E_j$, coming
from the other eigenspaces $E_k$ for $k\neq j$, via operators $\Sigma^1_{jk}:E_k\ra E_j$
such that $\bigl(\Sigma^1_{jk}\bigr)^*=\Sigma^1_{kj}$. In formula, we will have
$$
S^0\,:=\,\sum_{j}\Pi_j^*\,\Sigma^0_j\,\Pi_j\qquad\mbox{ and }\qquad
S^1\,:=\,\sum_j\sum_{k\neq j}\Pi_j^*\Sigma^1_{jk}\,\Pi_k\,.
$$
In both last two sentences above, ``suitable'' has to be read in function of the cancellations we want to produce in the energy estimates.
The first and main step is to find the $\Sigma^0_j$'s: we will reconduct this issue to the problem of solving a system of ODEs in low
regularity Zygmund classes. Notice that, in particular, we will be below the regularity required by classical existence theory for ODEs.
Nonetheless, by use of tools from Littlewood-Paley theory, we will be able to find \emph{approximate} solutions to our ODE system,
up to a (smooth) remainder which can be easily controlled in terms of the energy.
Once the $\Sigma^0_j$'s are built, a simple algebraic relation allow us to find also the $\Sigma^1_{jk}$'s.
We point out here that the constructed $S^0$ will have the same Zygmund regularity of the $A_j$'s, since the $\Sigma^0_j$'s will,
while $S^1$ will have one degree less of smoothness, because its definition will involve one time derivative of the original coeffcients.

Let us remark that adding a lower order term in the symmetrizer can be compared with the choice of Tarama,
in \cite{Tar}, of modifing the definition of the classical energy associated to a wave operator $W$
by introducing a lower order part. Nonetheless, we can transform the wave equation \eqref{intro_eq:wave}
into a system and apply the machinery we have just explained (see also Section \ref{s:wave}):
then, the ``system-energy'' doesn't coincide completely
with the original energy defined by Tarama, even if the result one can obtain is the same.
We note here that, in the particular case of systems coming from an equation, the integral condition \eqref{intro_hyp:int-LZ}
is enough to find the result (see also Remarks \ref{r:p} and \ref{r:p=1}).

At this point, it's interesting to notice also the analogy of our construction of the symmetrizer with the
two-steps diagonalization performed in \cite{C-DS-R} by Colombini, Del Santo and Reissig, still in dealing with the
wave operator $W$.

Let us conclude the introduction by pointing out that Zygmund classes can be characterized as special (possibly logarithmic) Besov spaces:
as mentioned before about the construction of the $\Sigma^0_j$'s, in our analysis
we will largely exploit Littlewood-Paley theory and ``logarithmic paradifferential calculus'' (see Section \ref{s:tools}).

\medbreak
Before going on, let us give a brief overview of the paper.

In the next section, we will give the basic definitions and we will state the main results, namely energy estimates with and without
loss and well-posedness of the Cauchy problem for $L$ in suitable Sobolev spaces.

In Section \ref{s:tools} we will introduce the tools, mainly from Fourier Analysis, we need in our study.
In particular, we will recall the basic points of the Littlewood-Paley theory, extending the classical construction to
logarithmic behaviours.

Section \ref{s:proof} is devoted to the proof of the statements. In particular, we will detail the construction
of a symmetrizer for our system and the computations in order to get energy estimates.

Finally, in Section \ref{s:wave} we will give a concrete example to illustrate our technique.
We will come back to the wave equation \eqref{intro_eq:wave}, we will transform it into a system and we will perform the analysis
we developed in the previous sections. At the end, we will recover the same results Tarama proved in \cite{Tar}, showing
however a slightly different proof.

In the Appendix we will postpone the proofs of some technical results.

\subsection*{Notations}
Before going on, let us introduce some notations.

First of all, given two vectors $v$ and $w$ in $\C^m$, we will denote by $v\cdot w$ the usual scalar product in $\C^m$ and
by $|v|$ the usual norm of a vector in $\C^m$:
$$
v\,\cdot\,w\,=\,\sum_{j=1}^m v_j\,\oline{w_j}\qquad\mbox{ and }\qquad
|v|^2\,=\,v\,\cdot\,v\,.
$$

On the contrary, given a infinite-dimensional Banach space $X$, we will denote by $\|\,\cdot\,\|_{X}$ its norm and,
if it's Hilbert, by $(\,\cdot\,,\,\cdot\,)_{X}$ its scalar product. Tipically, for us $X=L^2(\R^n;\R^m)$ or $H^s(\R^n;\R^m)$.

We will also set $\mc{M}_m(\R)$ the set of all $m\times m$ matrices whose components are real numbers, and we will denote
by $|\,\cdot\,|_{\mc{M}}$ its norm:
$$
|A|_{\mc{M}}\,:=\,\sup_{|v|=1}|Av|\,\equiv\,\sup_{|v|\leq1}|Av|\,\equiv\,\sup_{v\neq0}\frac{|Av|}{|v|}\,.
$$
With standard notations, we will denote by $D\,=\,{\rm diag}\left(d_1\ldots d_m\right)$ the diagonal matrix having as elements the
numbers $d_j$, and with $A\,=\,\left(v_1\,|\,\ldots\,|\,v_m\right)$ the matrix having $v_j$ as $j$-th column vector. We finally
set $B\,=\,\left(w_1\,-\,\dots\,-\,w_m\right)$ the matrix having $^tw_k$ as $k$-th line vector.

%%%%%%%%%%%%%%%%%%%%%%%%%%%%%%%%%%%%%%%%%%%%%%%%%%%%%%%%%
%%%%%%%%%%%%%%%%%%%%%%%%%%%%%%%%%%%%%%%%%%%%%%%%%%%%%%%%%
\section{Basic definitions and main results} \label{s:results}
%%%%%%%%%%%%%%%%%%%%%%%%%%%%%%%%%%%%%%%%%%%%%%%%%%%%%%%%%
%%%%%%%%%%%%%%%%%%%%%%%%%%%%%%%%%%%%%%%%%%%%%%%%%%%%%%%%%

For $m\geq1$, let us consider the $m\times m$ linear first order system
\begin{equation} \label{def:Lu}
Lu(t,x)\,=\,\d_tu(t,x)\,+\,\sum_{j=1}^nA_j(t)\,\d_ju(t,x)
\end{equation}
defined on a strip $[0,T]\times\R^n$, for some time $T>0$ and $n\geq1$.
We suppose $u(t,x)\in\R^m$ and, for all $1\leq j\leq n$, the matrices $A_j(t)\in\mc{M}_m(\R)$.

We define the symbol $A$ associated to the operator $L$: for all $(t,\xi)\in[0,T]\times\R^n$,
\begin{equation} \label{def:symbol}
A(t,\xi)\,:=\,\sum_{j=1}^n\xi_j\,A_j(t)\,.
\end{equation}
Then, for all $(t,\xi)$, $A(t,\xi)$ is an $m\times m$ matrix which has real-valued coefficients.
We denote by $\bigl(\lambda_j(t,\xi)\bigr)_{1\leq j\leq m}\subset\C$ its eigenvalues at any point $(t,\xi)$.

\begin{rem} \label{r:homog}
Note that $A(t,\xi)$ is homogeneous of degree $1$ in $\xi$, and this property is inherited by the eigenvalues.
As a matter of facts, for any $\g>0$,
$$
A(t,\xi)v(t,\xi)\,=\,\lambda(t,\xi) v(t,\xi)\qquad\Longrightarrow\qquad A(t,\g\xi)v(t,\xi)\,=\,\g\lambda(t,\xi)v(t,\xi)\,;
$$
this relation shows in particular that
$$
v(t,\g\xi)\,=\,v(t,\xi)\qquad\mbox{ and }\qquad \lambda(t,\g\xi)\,=\,\g\lambda(t,\xi)\,.
$$
\end{rem}

Let us introduce the following definitions (see e.g. \cite{M-2008}, Chapter 2).
\begin{defin} \label{d:systems}
\begin{itemize}
\item[(i)] We say that system \eqref{def:Lu} is \emph{strictly hyperbolic} if, for all $t\in[0,T]$ and all $\xi\neq0$,
the eigenvalues of $A(t,\xi)$ are all real and distinct:
$$
\bigl(\lambda_j\bigr)_{1\leq j\leq m}\,\subset\,\R\qquad\qquad\mbox{ and }\qquad\qquad
\lambda_j\,\neq\,\lambda_k\quad\mbox{ for }\;j\neq k\,.
$$

\item[(ii)] System \eqref{def:Lu} is said instead to be \emph{hyperbolic with constant multiplicities} if, for all $t\in[0,T]$ and all
$\xi\neq0$, the eigenvalues of $A(t,\xi)$ are real and semi-simple, with constant multiplicities.
\end{itemize}
\end{defin}

We recall that a (possibly complex) eigenvalue is called \emph{semi-simple} if its algebraic and geometric multiplicities coincide;
a matrix is semi-simple if it is diagonalizable in the complex sense. Then, assuming the system
to be hyperbolic with constant multiplicities means that $A(t,\xi)$ is diagonalizable at any point $(t,\xi)$, its eigenvalues
are real and their multiplicities don't change in $t$ nor in $\xi$.

\medbreak
We will always assume our system to be hyperbolic with constant multiplicities. Let us note that, in particular, under this hypothesis we
have $\bigl(\lambda_j(t,\xi)\bigr)_{1\leq j\leq m}\subset\R$.

Let us turn our attention to the coefficients of $L$. In the whole paper, we will suppose that, for all $1\leq j\leq n$,
the matrix-valued functions $A_j$ belong to $L^\infty$:
\begin{equation} \label{hyp:bound}
\bigl\|A_j\bigr\|_{L^\infty([0,T];\mc{M}_m(\R))}\,:=\,\sup_{[0,T]}\bigl|A_j(t)\bigr|_{\mc{M}}\,\leq\,K_0\,.
\end{equation}

In a first time, let us assume that they satisfy a Zygmund regularity condition: there exist a $p\in[1,+\infty]$ and a constant $K_z>0$ such
that, for all $1\leq j \leq n$ and all $0<\tau<T/2$,
\begin{equation} \label{hyp:Z}
\bigl\|A_j(\,\cdot\,+\tau)\,+\,A_j(\,\cdot\,-\tau)\,-\,2\,A_j(\,\cdot\,)\bigr\|_{L^p([\tau,T-\tau];\mc{M}_m(\R))}\,\leq\,K_z\,\tau\,.
\end{equation}
Note that this condition tells us that each component of the matrices $A_j$ verifies the same integral Zygmund condition (as
real-valued functions on $[0,T]$).

Let us point out that, as we will see in Section \ref{s:tools}, if $p>1$ then condition \eqref{hyp:Z}
already implies the boundedness property \eqref{hyp:bound}.

Under this hypothesis, it's possible to prove an energy estimate with no loss of derivatives for our operator $L$.

\begin{thm} \label{th:en_Z}
Let us consider the first-order system \eqref{def:Lu}, and let us assume it to be hyperbolic with constant multiplicities.
Suppose moreover that the coefficients $\bigl(A_j\bigr)_{1\leq j\leq n}$ satisfy
the Zygmund condition \eqref{hyp:Z}, for some $p\in\,]1,+\infty]$.

Then, for all $s\in\R$, there exist positive constants $C_1$, $C_2$ (just depending on $s$
and on $K_z$) such that the estimate
\begin{equation} \label{est:u_Z}
\sup_{t\in[0,T]}\|u(t)\|_{H^s}\,\leq\,C_1\,e^{C_2\,T}\left(\|u(0)\|_{H^s}\,+\,\int^T_0\bigl\|Lu(\tau)\bigr\|_{H^s}\,d\tau\right)
\end{equation}
holds true for any $u\in\mc{C}^1([0,T];H^\infty(\R^n;\R^m))$.
\end{thm}

\begin{rem} \label{r:p}
We point out that in the previous statement, as well as in Theorem \ref{th:en_LZ} below, in general $p$ has to be strictly greater
than $1$. We will be clearer about this point in Subsection \ref{ss:symm}, where we will construct a symmetrizer for $L$.

However, as we will see in Section \ref{s:wave}, in the particular case of systems coming from a second order scalar equation,
the weakest condition $p=1$, combined with the additional $L^\infty$ assumption \eqref{hyp:bound},
is still suitable to recover energy estimates.
\end{rem}

From the previous result, it immediately follows the well-posedness issue of the Cauchy problem related to operator $L$.
\begin{thm} \label{th:cauchy_Z}
Under the hypothesis of Theorem \ref{th:en_Z}, for any $s\in\R$ the Cauchy problem for $L$,
$$
\left\{\begin{array}{l}
        Lu\;=\;f \\[1ex]
	u_{|t=0}\;=\;u_0\,,
       \end{array}\right. \leqno{(C\!P)}
$$
is well-posed in $H^s(\R^n;\R^m)$, globally on $[0,T]$.

In particular, $(C\!P)$ is well-posed in the space $H^\infty$ with no loss of derivatives.
\end{thm}

Now, let us consider the weaker log-Zygmund condition: there are a $p\in[1,+\infty]$ and a positive constant $K_{\ell z}$
such that, for all $1\leq j \leq n$ and all $0<\tau<T/2$, one has
\begin{equation} \label{hyp:LZ}
\bigl\|A_j(\,\cdot\,+\tau)\,+\,A_j(\,\cdot\,-\tau)\,-\,2\,A_j(\,\cdot\,)\bigr\|_{L^p([\tau,T-\tau];\mc{M}_m(\R))}\,\leq\,
K_{\ell z}\,\tau\,\log\left(1+\frac{1}{\tau}\right)\,.
\end{equation}
Again this condition tells us that each component of the matrices $A_j$ verifies the same integral log-Zygmund condition.

Under this lower regularity assumption, we are able to prove an energy estimate with a finite loss of derivatives.
\begin{thm} \label{th:en_LZ}
Let us consider the first-order system \eqref{def:Lu}, and assume it to be hyperbolic with constant multiplicities.
Suppose moreover that the coefficients $\bigl(A_j\bigr)_{1\leq j\leq n}$ satisfy
the log-Zygmund condition \eqref{hyp:LZ}, for some $p\in\,]1,+\infty]$.

Then, for all $s\in\R$, there exist a positive constants $C_1$, $C_2$ (depending on $s$
and on $K_{\ell z}$) and a $\wtilde{\beta}>0$ (depending just on 
$K_{\ell z}$) such that, setting
$$
\beta(t)\,:=\,\wtilde{\beta}\,t^\gamma\,,\qquad\qquad\mbox{ with }\qquad\g\,=\,\frac{1}{p'}\,=\,1\,-\,\frac{1}{p}\,,
$$
then the estimate
\begin{equation} \label{est:u_LZ}
\sup_{t\in[0,T]}\|u(t)\|_{H^{s-\beta(t)}}\,\leq\,C_1\,e^{C_2\,T}\,\left(\|u(0)\|_{H^s}\,+\,\int^T_0
\bigl\|Lu(\tau)\bigr\|_{H^{s-\beta(T)+\beta(T-\tau)}}\,d\tau\right)
\end{equation}
holds true for any $u\in\mc{C}^1\bigl([0,T];H^\infty(\R^n;\R^m)\bigr)$.
\end{thm}

Let us stress again the fact that, in general, we do have $p>1$. Moreover, in this instance
condition \eqref{hyp:LZ} is still enough to recover the boundedness of the coefficients \eqref{hyp:bound}.

\begin{rem} \label{r:th_loss}
 The loss in the right-hand side of \eqref{est:u_LZ} is due to the $L^p$ hypothesis with logarithmic behaviour \eqref{hyp:LZ},
and it appears a little bit strange, also because it is not suitable for iterations in time, apart from the case
$p=+\infty$ (see Theorem \ref{th:en_LZ_inf} below).

We will come back later on this question (see Remark \ref{r:loss}) when proving the previous estimate.
For the moment, let us note that, if related to the case of scalar second order hyperbolic equations, such a loss represents the
intermediate instance between the fixed loss under an integral log-Zygmund condition (see paper \cite{Tar}) and the time-dependent
one for its pointwise counterpart (see e.g. \cite{C-DS} and \cite{C-DS-F-M_tl}).
\end{rem}

As noticed in the previous remark, in the case $p=+\infty$ one can recover a linear in time loss of derivatives in
the energy estimates. Such an estimate is suitable for iterations in time.

\begin{thm} \label{th:en_LZ_inf}
 Let us consider the first-order system \eqref{def:Lu}, and assume it to be hyperbolic with constant multiplicities.
Suppose moreover that the coefficients $\bigl(A_j\bigr)_{1\leq j\leq n}$ satisfy
the log-Zygmund condition \eqref{hyp:LZ} for $p=+\infty$.

Then, for all $s\in\R$, there exist positive constants $C_1$, $C_2$ (depending on $s$
and on $K_{\ell z}$) and a $\wtilde{\beta}>0$ (depending just on 
$K_{\ell z}$) such that the estimate
\begin{equation} \label{est:u_LZ_inf}
\sup_{t\in[0,T]}\|u(t)\|_{H^{s-\wtilde{\beta}t}}\,\leq\,C_1\,e^{C_2\,T}\,\left(\|u(0)\|_{H^s}\,+\,\int^T_0
\bigl\|Lu(\tau)\bigr\|_{H^{s-\wtilde{\beta}\tau}}\,d\tau\right)
\end{equation}
holds true for any $u\in\mc{C}^1\bigl([0,T];H^\infty(\R^n;\R^m)\bigr)$.
\end{thm}

From Theorem \ref{th:en_LZ}, it immediately follows the well-posedness issue of the Cauchy problem
related to the operator $L$, but just in the space $H^\infty$, due to the finite loss of derivatives.
\begin{thm} \label{th:cauchy_LZ}
Under the hypothesis of Theorem \ref{th:en_LZ}, the Cauchy problem $(C\!P)$ for $L$
is well-posed in the space $H^\infty(\R^n;\R^m)$ with a finite loss of derivatives.
\end{thm}

%%%%%%%%%%%%%%%%%%%%%%%%%%%%%%%%%%%%%%%%%%%%%%%%%%%%%%%%%
%%%%%%%%%%%%%%%%%%%%%%%%%%%%%%%%%%%%%%%%%%%%%%%%%%%%%%%%%
\section{Tools} \label{s:tools}
%%%%%%%%%%%%%%%%%%%%%%%%%%%%%%%%%%%%%%%%%%%%%%%%%%%%%%%%%
%%%%%%%%%%%%%%%%%%%%%%%%%%%%%%%%%%%%%%%%%%%%%%%%%%%%%%%%%

We collect here some notions and results which turn out to be useful in our proof. 

In a first time we will recall some basic facts on Littlewood-Paley theory. For the sake of completeness we will work on the
general instance of $\R^d$, with $d\geq1$.

Then, for reasons which will be clear in Subsection \ref{ss:Zygmund}, we will need to introduce the class of ``logarithmic Besov spaces''
and to develop paradifferential calculus in this framework. It must be said, however, that this study involves no
special difficulties, and it can be performed as in the classical case.

This having been done, we will focus on the instance $d=1$: we will introduce the Zygmund classes and we will provide a full
characterization for them by use of the previous tools.
Fundamental results about solving ODEs in Zygmund spaces will close this section.

%%%%%%%%%%%%%%%%%%%%%%%%%%%%%%%%%%%%%%%%%%%%%%%%%%%%%%%%%
\subsection{Littlewood-Paley theory} \label{ss:L-P}
%%%%%%%%%%%%%%%%%%%%%%%%%%%%%%%%%%%%%%%%%%%%%%%%%%%%%%%%%

Let us first define the so called ``Littlewood-Paley decomposition'', based on a non-homogeneous dyadic partition of unity with
respect to the Fourier variable. We refer to \cite{B-C-D} (Chapter 2), paper \cite{Bony} and \cite{M-2008} (Chapters 4 and 5)
for the details.

So, fix a smooth radial function
$\chi$ supported in the ball $B(0,2),$ 
equal to $1$ in a neighborhood of $B(0,1)$
and such that $r\mapsto\chi(r\,e)$ is nonincreasing
over $\R_+$ for all unitary vectors $e\in\R^d$. Set
$\varphi\left(\xi\right)=\chi\left(\xi\right)-\chi\left(2\xi\right)$ and
$\vphi_j(\xi):=\vphi(2^{-j}\xi)$ for all $j\geq0$.

The dyadic blocks $(\Delta_j)_{j\in\Z}$
 are defined by\footnote{Throughout we agree  that  $f(D)$ stands for 
the pseudo-differential operator $u\mapsto\mc{F}^{-1}(f\,\mc{F}u)$.} 
$$
\Delta_j:=0\ \hbox{ if }\ j\leq-1,\quad\Delta_{0}:=\chi(D)\quad\hbox{and}\quad
\Delta_j:=\varphi(2^{-j}D)\ \text{ if }\  j\geq1.
$$
We  also introduce the following low frequency cut-off:
$$
S_ju\,:=\,\chi(2^{-j}D)\,u\,=\,\sum_{k\leq j}\Delta_{k}u\quad\text{for}\quad j\geq0.
$$
Throughout the paper we will use freely the following classical property:
for any $u\in\mc{S}',$ the equality $u=\sum_{j}\Delta_ju$ holds true in $\mc{S}'$.

Let us also mention the so-called \emph{Bernstein's inequalities}, which explain
the way derivatives act on spectrally localized functions.
  \begin{lemma} \label{l:bern}
Let  $0<r<R$.   A
constant $C$ exists so that, for any nonnegative integer $k$, any couple $(p,q)$ 
in $[1,+\infty]^2$ with  $p\leq q$ 
and any function $u\in L^p$,  we  have, for all $\lambda>0$,
$$
\displaylines{
{\rm supp}\, \widehat u \subset   B(0,\lambda R)\quad
\Longrightarrow\quad
\|\nabla^k u\|_{L^q}\, \leq\,
 C^{k+1}\,\lambda^{k+d\left(\frac{1}{p}-\frac{1}{q}\right)}\,\|u\|_{L^p}\;;\cr
{\rm supp}\, \widehat u \subset \{\xi\in\R^d\,|\, r\lambda\leq|\xi|\leq R\lambda\}
\quad\Longrightarrow\quad C^{-k-1}\,\lambda^k\|u\|_{L^p}\,
\leq\,
\|\nabla^k u\|_{L^p}\,
\leq\,
C^{k+1} \, \lambda^k\|u\|_{L^p}\,.
}$$
\end{lemma}   

Let us recall the characterization of (classical) Sobolev spaces via dyadic decomposition:
for all $s\in\mbb{R}$ there exists a constant $C_s>0$ such that
\begin{equation} \label{est:dyad-Sob}
 \frac{1}{C_s}\,\,\sum^{+\infty}_{\nu=0}2^{2\nu s}\,\|u_\nu\|^2_{L^2}\;\leq\;\|u\|^2_{H^s}\;\leq\;
C_s\,\,\sum^{+\infty}_{\nu=0}2^{2\nu s}\,\|u_\nu\|^2_{L^2}\,,
\end{equation}
where we have set $u_\nu:=\Delta_\nu u$.

This property was then generalized in \cite{C-M} to logarithmic Sobolev spaces, which naturally come into play in the study
of hyperbolic operators with low regularity coefficients (at this purpose, see also \cite{C-DS-F-M_tl} and \cite{C-DS-F-M_wp}).

Let us set $\Pi(D)\,:=\,\log(2+|D|)$, i.e. its symbol is $\pi(\xi)\,:=\,\log(2+|\xi|)$.
\begin{defin} \label{d:log-H^s}
 For all $\alpha\in\R$, we define the space $H^{s+\alpha\log}$ as the space $\Pi^{-\alpha}H^s$, i.e.
$$
f\,\in\,H^{s+\alpha\log}\quad\Longleftrightarrow\quad\Pi^\alpha f\,\in\,H^s\quad\Longleftrightarrow\quad
\pi^\alpha(\xi)\left(1+|\xi|^2\right)^{s/2}\what{f}(\xi)\,\in\,L^2\,.
$$
\end{defin}

We have the following dyadic characterization of these spaces (see \cite{M-2008}, Proposition 4.1.11),
which generalizes property \eqref{est:dyad-Sob}.
\begin{prop} \label{p:log-H}
 Let $s$, $\alpha\,\in\R$. A $u\in\mc{S}'$ belongs to the space $H^{s+\alpha\log}$ if and only if:
\begin{itemize}
 \item[(i)] for all $k\in\N$, $\Delta_ku\in L^2(\R^d)$;
\item[(ii)] set $\,\delta_k\,:=\,2^{ks}\,(1+k)^\alpha\,\|\Delta_ku\|_{L^2}$ for all $k\in\N$, the sequence
$\left(\delta_k\right)_k$ belongs to $\ell^2(\N)$.
\end{itemize}
Moreover, $\|u\|_{H^{s+\alpha\log}}\,\sim\,\left\|\left(\delta_k\right)_k\right\|_{\ell^2}$.
\end{prop}

It turns out that also Zygmund classes can be characterized in terms of Littlewood-Paley decomposition as particular Besov spaces
(see the next subsection). We will broadly exploit this fact in proving our results.

However, for reasons which appear clear in the sequel, we need to introduce logarithmic functional spaces, in the same spirit of the ones
of Definition \ref{d:log-H^s}, and to develop paradifferential calculus in this new framework.

%%%%%%%%
\subsubsection{Logarithmic Besov spaces}
%%%%%%%%

We introduce now the class of logarithmic Besov spaces.
We quote and prove here just the basic results we will need in our computations;
we refer to \cite{F_phd} for more properties. Let us point out that, essentially, everything works as in the classical case,
with just slight extensions of the statements and slight modifications in the arguments of the proofs.

We start with a definition.
\begin{defin} \label{d:log-B}
  Let $s$ and $\alpha$ be real numbers, and $1\leq p,r\leq+\infty$. The \emph{non-homogeneous logarithmic Besov space}
$B^{s+\alpha\log}_{p,r}$ is defined as the subset of tempered distributions $u$ for which
$$
\|u\|_{B^{s+\alpha\log}_{p,r}}\,:=\,
\left\|\left(2^{js}\,(1+j)^{\alpha}\,\|\Delta_ju\|_{L^p}\right)_{j\in\N}\right\|_{\ell^r}\,<\,+\infty\,.
$$
\end{defin}

First of all, let us show that the previous definition is independent of the choice of the cut-off functions defining the Littlewood-Paley
decomposition.

\begin{lemma} \label{l:log-B_ind}
 Let $\mc{C}\subset\R^d$ be a ring, $(s,\alpha)\in\R^2$ and $(p,r)\in[1,+\infty]^2$. Let $\left(u_j\right)_{j\in\N}$ be
a sequence of smooth functions such that
$$
{\rm supp}\,\what{u}_j\,\subset\,2^j\,\mc{C}\qquad\quad\mbox{ and }\qquad\quad
\left\|\left(2^{js}\,(1+j)^\alpha\,\|u_j\|_{L^p}\right)_{j\in\N}\right\|_{\ell^r}\,<\,+\infty\,.
$$

Then $u:=\sum_{j\in\N}u_j$ belongs to $B^{s+\alpha\log}_{p,r}$ and
$$
\|u\|_{B^{s+\alpha\log}_{p,r}}\,\leq\,C_{s,\alpha}\,\left\|\left(2^{js}\,(1+j)^\alpha\,
\|u_j\|_{L^p}\right)_{j\in\N}\right\|_{\ell^r}\,.
$$
\end{lemma}

\begin{proof}
 By spectral localization, we gather that there exists a $n_0\in\N$ such that $\Delta_ku_j=0$ for all $|k-j|>n_0$. Therefore
$$
\|\Delta_ku\|_{L^p}\,\leq\,\sum_{|j-k|\leq n_0}\|\Delta_ku_j\|_{L^p}\,\leq\,C\,\sum_{|j-k|\leq n_0}\|u_j\|_{L^p}\,.
$$
From these relations it immediately follows that
$$
 2^{ks}\,(1+k)^\alpha\,\|\Delta_ku\|_{L^p}\,\leq\,C\,\sum_{|j-k|\leq n_0}2^{(k-j)s}\,
\frac{(1+k)^\alpha}{(1+j)^\alpha}\,2^{js}\,(1+j)^\alpha\,\|u_j\|_{L^p}\,.
$$
Now, as very often in the sequel, we use the fact that
\begin{equation} \label{est:frac_int}
\frac{(1+k)}{(1+j)}\,\leq\,1+|k-j|\,.
\end{equation}
Hence, we get
$$
 2^{ks}\,(1+k)^\alpha\,\|\Delta_ku\|_{L^p}\,\leq\,C\,(\theta *\delta)_k\,,
$$
where we have set (here $\mc{I}_A$ denote the characteristic function of the set $A$)
$$
\theta_h\,:=\,2^{hs}\,(1+h)^{|\alpha|}\,\mc{I}_{[0,n_0]}(h)\qquad\mbox{ and }\qquad
\delta_j\,:=\,2^{js}\,(1+j)^\alpha\,\|u_j\|_{L^p}\,.
$$
Passing to the $\ell^r$ norm and applying Young's inequality for convolutions complete the proof.
\end{proof}

So, Definition \ref{d:log-B} makes sense.
Of course, for $\alpha=0$ we get the classical Besov classes $B^s_{p,r}$.

Recall that, for all $s\in\,\R_+\!\!\setminus\!\N$, the space $B^s_{\infty,\infty}$ coincides with the H\"older space $\mc{C}^s$.
If $s\in\N$, instead, we set $\mc{C}^s_*:=B^s_{\infty,\infty}$, to distinguish it from the space $\mc{C}^s$ of
the differentiable functions with continuous partial derivatives up to the order $s$. Moreover, the strict inclusion
$\mc{C}^s_b\,\hra\,\mc{C}^s_*$ holds, where $\mc{C}^s_b$ denotes the subset of $\mc{C}^s$ functions bounded with all
their derivatives up to the order $s$.
Finally, for $s<0$, the ``negative H\"older space'' $\mc{C}^s$ is defined as the Besov space $B^s_{\infty,\infty}$.

  Let us point out that for any $k\in\N$ and $p\in[1,+\infty]$, we have the following chain of continuous embeddings:
 $$
 B^k_{p,1}\hookrightarrow W^{k,p}\hookrightarrow B^k_{p,\infty}\,,
 $$
  where  $W^{k,p}$ denotes the classical Sobolev space of $L^p$ functions
 with all the derivatives up to the order $k$ in $L^p$. However, for all $s\in\R$, we have the equivalence $B^s_{2,2}\equiv H^s$,
 as stated by relation \eqref{est:dyad-Sob}, and Proposition \ref{p:log-H} tells us that this is still true when considering
 the logarithmic case.

Generally speaking, logarithmic Besov spaces are intermediate classes of functions between the classical ones. As a matter of fact, we
have the following result.
\begin{prop} \label{p:log-emb}
The space $B^{s_1+\alpha_1\log}_{p_1,r_1}$ is continuously embedded in the space $B^{s_2+\alpha_2\log}_{p_2,r_2}$ whenever
$\,1\,\leq\,p_1\,\leq\,p_2\,\leq\,+\infty$ and one of the following conditions holds true:
\begin{itemize}
\item $s_2\,=\,s_1\,-\,d\,(1/p_1\,-\,1/p_2)\,$, $\,\alpha_2\,\leq\,\alpha_1\,$ and $\,1\,\leq\,r_1\,\leq\,r_2\,\leq\,+\infty\,$;
\item $s_2\,=\,s_1\,-\,d\,(1/p_1\,-\,1/p_2)\,$ and $\,\alpha_1\,-\,\alpha_2\,>\,1\,$;
\item $s_2\,<\,s_1\,-\,d\,(1/p_1\,-\,1/p_2)\,$.
\end{itemize}
\end{prop}

\begin{proof}
As in the classical case, these properties are straightforward consequences of Bernstein's inequalities.
As a matter of fact, considering for a while the first instance, and just the case $r_1=r_2=1$ thanks to
the embeddings of $\ell^r$ spaces, we can write
\begin{equation} \label{est:embedd}
\sum_{j=0}^{+\infty}2^{js_2}\,(1+j)^{\alpha_2}\,\|\Delta_ju\|_{L^{p_2}}\,\leq\,
C\sum_{j=0}^{+\infty}2^{js_1}\,(1+j)^{\alpha_1}\,\|\Delta_ju\|_{L^{p_1}}\,(1+j)^{\alpha_2-\alpha_1}\,2^{j\wtilde{s}}\,,
\end{equation}
where we have set
$$
\wtilde{s}\,=\,s_2\,-\,s_1\,-\,d\left(\frac{1}{p_1}\,-\,\frac{1}{p_2}\right)\,.
$$
Now, in the first instance we have $\wtilde{s}=0$ and $\alpha_2-\alpha_1\leq0$, and the conclusion follows.

For the proof of the second part, it's enough to consider the endpoint case $r_2=1$, $r_1=+\infty$.
Again, the result issues from \eqref{est:embedd}, with $\wtilde{s}=0$ and
$\alpha_1-\alpha_2>1$.

The last sentence can be proved in the same way, again in the limit instance $r_2=1$, $r_1=+\infty$, noting that $\wtilde{s}<0$,
and this behaviour is stronger than the logarithmic one.
\end{proof}

Now we want to consider the action of Fourier multipliers on non-homogeneous logarithmic Besov spaces.
First of all, we have to give a more general definition of symbols.
\begin{defin} \label{d:log-mult}
 A smooth function $f:\R^d\longrightarrow\R\,$ is said to be a \emph{$S^{m+\delta\log}$-multiplier} if,
for all multi-index $\nu\in\N^d$,
there exists a constant $C_\nu$ such that
$$
\forall\;\xi\in\R^d\,,\qquad\left|\d^\nu_\xi f(\xi)\right|\,\leq\,C_\nu\,\bigl(1+|\xi|\bigr)^{m-|\nu|}\,
\log^\delta\bigl(1+|\xi|\bigr)\,.
$$
\end{defin}

\begin{prop} \label{p:log-mult}
 Let $m,\delta\,\in\R$ and $f$ be a $S^{m+\delta\log}$-multiplier.

Then for all real numbers $s$ and $\alpha$ and all  $(p,r)\in[1,+\infty]^2$, the operator $f(D)$ maps
$B^{s+\alpha\log}_{p,r}$ into $B^{(s-m)+(\alpha-\delta)\log}_{p,r}$ continuously.
\end{prop}

\begin{proof}
 According to Lemma \ref{l:log-B_ind}, it's enough to prove that, for all $j\geq0$,
$$
2^{(s-m)j}\,(1+j)^{\alpha-\delta}\,\left\|f(D)\,\Delta_ju\right\|_{L^p}\,\leq\,
C\,2^{js}\,(1+j)^\alpha\,\left\|\Delta_ju\right\|_{L^p}\,.
$$

Let us deal with low frequencies first. Take a $\theta\in\mc{D}(\R^d)$ such that $\theta\equiv1$ in a neighborhood
of ${\rm supp}\chi$: passing to the phase space, it's easy to see that $f(D)\,\Delta_0u\,=\,(\theta f)(D) \Delta_0u$.
As $\mc{F}^{-1}(\theta f)\in L^1$, Young's inequality for convolutions gives us the desired estimate for $j=0$.

Now we focus on high frequencies and we fix a $j\geq1$. Noticing that the function
$\wtilde{\vphi}_j:=\vphi_{j-1}+\vphi_j+\vphi_{j+1}$ is equal to $1$ on the support of $\vphi_j$, with easy computations
we get the equality
$$
f(D)\,\Delta_ju\;=\;2^{jd}\,F_j(2^j\,\cdot\,)\,*\,\Delta_ju\,,
$$
where, denoted $\wtilde{\vphi}=\vphi_0+\vphi_1+\vphi_2$, we have set
$$
F_j(x)\;=\;\frac{1}{(2\pi)^d}\,\int_{\R^d}e^{ix\cdot\xi}\,f(2^j\xi)\,\wtilde{\vphi}(\xi)\,d\xi\,.
$$
Let us prove that $F_j\in L^1$. For all $N\in\N$, we can write
\begin{eqnarray*}
(1+|x|^2)^N\,F_j(x) & = & \frac{1}{(2\pi)^d}\int_{\R^d_\xi}e^{ix\cdot\xi}\left(\Id-\Delta_\xi\right)^N\biggl(f(2^j\xi)\,
\wtilde{\vphi}(\xi)\biggr)\,d\xi \\
& = & \sum_{|\beta|+|\gamma|\leq2N}\frac{C_{\beta,\gamma}}{(2\pi)^d}\,2^{j|\beta|}
\int_{\R^d_\xi}e^{ix\cdot\xi}\,(\d^\beta f)(2^j\xi)\,(\d^\gamma\wtilde{\vphi})(\xi)\,d\xi\,.
\end{eqnarray*}
In fact, the integration is not performed on the whole $\R^d$, but only on the support of $\wtilde{\vphi}$, which is
a ring $\mc{R}\,:=\,\left\{c\,\leq\,|\xi|\,\leq\,C\right\}$, independent of $j$. Therefore, recalling the definition
of $S^{m+\delta\log}$-multiplier, we gather
$$
(1+|x|^2)^N\,|F_j(x)|\,\leq\,
C_{d,N}\,2^{jm}\,(1+j)^\delta\,,
$$
which implies that, for $N$ big enough, $F_j\in L^1(\R^d_x)$ and $\|F_j\|_{L^1}\,\leq\,C\,2^{mj}\,(1+j)^\delta$. Young's
inequality for convolution leads then to the result.
\end{proof}

Let us conclude this part with two technical lemmas, which will be immediately useful in the next paragraph.

We start with a characterization of logarithmic Besov spaces in terms of the low frequencies cut-off operators.
This will be relevant in analysing continuity properties of the paraproduct operator.

\begin{lemma} \label{l:log-S_j}
Fix $(s,\alpha)\in\R^2$ and $(p,r)\in[1,+\infty]^2$, and let $u\in\mc{S}'$ given.
\begin{itemize}
\item[(i)] If the sequence $\bigl(2^{js}\,(1+j)^\alpha\,\|S_ju\|_{L^p}\bigr)_{j\in\N}$ belongs to $\ell^r$,
then $u\in B^{s+\alpha\log}_{p,r}$ and
$$
\|u\|_{B^{s+\alpha\log}_{p,r}}\;\leq\;C\,\left\|\bigl(2^{js}\,(1+j)^\alpha\,\|S_ju\|_{L^p}\bigr)_{j\in\N}\right\|_{\ell^r}\,,
$$
for some constant $C>0$ depending only on $s$ and $\alpha$, but not on $u$.
\item[(ii)] Suppose $u\in B^{s+\alpha\log}_{p,r}$, with $s<0$. Then the sequence
$\bigl(2^{js}\,(1+j)^\alpha\,\|S_ju\|_{L^p}\bigr)_{j\in\N}\,\in\,\ell^r$, and
$$
\left\|\bigl(2^{js}\,(1+j)^\alpha\,\|S_ju\|_{L^p}\bigr)_{j\in\N}\right\|_{\ell^r}\;\leq\;
\wtilde{C}\,\|u\|_{B^{s+\alpha\log}_{p,r}}\,,
$$
for some constant $\wtilde{C}>0$ depending only on $s$ and $\alpha$.
\item[(iii)] In the endpoint case $s=0$, one can only infer, for any $\alpha\leq0$,
$$
\left\|\biggl((1+j)^\alpha\,\|S_ju\|_{L^p}\biggr)_{j\in\N}\right\|_{\ell^\infty}\;\leq\;
\wtilde{C}\,\|u\|_{B^{0+\alpha\log}_{p,1}}\,.
$$
\end{itemize}
\end{lemma}

\begin{proof}
 From the definitions, we have $\Delta_j\,=\,S_{j+1}-S_j$. So we can write:
\begin{eqnarray*}
 2^{js}\,(1+j)^\alpha\,\left\|\Delta_ju\right\|_{L^p} & \leq & 2^{js}\,(1+j)^\alpha\left(\left\|S_{j+1}u\right\|_{L^p}\,+\,
\left\|S_ju\right\|_{L^p}\right) \\[1ex]
& \leq & 2^{(j+1)s}\,(2+j)^\alpha\,\left\|S_{j+1}u\right\|_{L^p}\,\frac{(1+j)^\alpha}{(2+j)^\alpha}\,2^{-s}\,+\,
2^{js}\,(1+j)^\alpha\left\|S_ju\right\|_{L^p}\,.
\end{eqnarray*}
By Minowski's inequality, we get the first part of the statement.

On the other hand, using the definition of the operator $S_j$, we have
\begin{eqnarray*}
 2^{js}\,(1+j)^\alpha\,\left\|S_ju\right\|_{L^p} & \leq & 2^{js}\,(1+j)^\alpha
\sum_{k\leq j-1}\left\|\Delta_ku\right\|_{L^p} \\
& \leq & \sum_{k\leq j-1}2^{(j-k)s}\,\frac{(1+j)^\alpha}{(1+k)^\alpha}\,2^{ks}\,(1+k)^\alpha\,\left\|\Delta_ku\right\|_{L^p} \\[1ex]
& \leq & C\,\left(\theta\,*\,\delta\right)_j\,,
\end{eqnarray*}
where we have argued as in proving Lemma \ref{l:log-B_ind}, setting
$$
\theta_h\,:=\,2^{hs}\,(1+h)^{|\alpha|}\qquad\mbox{ and }\qquad \delta_k\,:=\,2^{ks}\,(1+k)^\alpha\,\left\|\Delta_ku\right\|_{L^p}\,.
$$
Then, in the case $s<0$, the sequence
 $\left(\theta_h\right)_h\in\ell^1$; hence, Young's inequality for
convolution gives us the result.

For the case $s=0$, $\alpha\leq0$, we argue as before and we write
\begin{eqnarray*}
(1+j)^\alpha\,\left\|S_ju\right\|_{L^p} & = &
\sum_{k\leq j-1}\frac{(1+k)^{-\alpha}}{(1+j)^{-\alpha}}\,(1+k)^\alpha\,\left\|\Delta_ku\right\|_{L^p} \\
& \leq & \sum_{k\leq j-1}\frac{j^{-\alpha}}{(1+j)^{-\alpha}}\,(1+k)^\alpha\,\left\|\Delta_ku\right\|_{L^p}\,;
\end{eqnarray*}
this relation allows us to conlude, passing to the $\ell^\infty$ norm with respect to $j$.
\end{proof}

The second lemma, instead, will be useful for the analysis of the remainder operator in the Bony's paraproduct decomposition.

\begin{lemma} \label{l:log-ball}
 Let $\mc{B}$ be a ball of $\R^d$, and the couple $(p,r)$ belong to $[1,+\infty]^2$.
Let $s>0$ and $\alpha\in\R$.
Let $\left(u_j\right)_{j\in\N}$ be a sequence of smooth functions such that
$$
{\rm supp}\,\what{u}_j\,\subset\,2^j\mc{B}\qquad\mbox{ and }\qquad
\bigl(2^{js}\,(1+j)^\alpha\,\left\|u_j\right\|_{L^p}\bigr)_{j\in\N}\,\in\,\ell^r\,.
$$

Then the function $\,u\,:=\,\sum_{j\in\N}u_j\,$ belongs to the space $B^{s+\alpha\log}_{p,r}$. Moreover, there exists a constant
$C$, depending only on $s$ and $\alpha$, such that
$$
\|u\|_{B^{s+\alpha\log}_{p,r}}\,\leq\,C\,\left\|\left(2^{js}\,(1+j)^\alpha\,
\left\|u_j\right\|_{L^p}\right)_{j\in\N}\right\|_{\ell^r}\,.
$$

In the endpoint case $s=0$, one can just infer, for any $\alpha\geq0$,
$$
\|u\|_{B^{0+\alpha\log}_{p,\infty}}\,\leq\,C\,\left\|\biggl((1+j)^\alpha\,\left\|u_j\right\|_{L^p}\biggr)_{j\in\N}\right\|_{\ell^1}\,.
$$
\end{lemma}

\begin{proof}
 We have to estimate $\|\Delta_ku\|_{L^p}\,\leq\,\sum_{j}\|\Delta_ku_j\|_{L^p}$.

From our hypothesis on the support
of each $\what{u}_j$, we infer that there exists an index $n_0\in\N$ such that $\Delta_ku_j\equiv0$ for all
$k>j+n_0$. Therefore, arguing as already done in previous proofs,
\begin{eqnarray*}
2^{ks}\,(1+k)^\alpha\,\|\Delta_ku\|_{L^p} & \leq & \sum_{j\geq k-n_0}2^{(k-j)s}\,\frac{(1+k)^\alpha}{(1+j)^\alpha}\,
2^{js}\,(1+j)^\alpha\,\left\|u_j\right\|_{L^p} \\
& \leq & \sum_{j\geq k-n_0}2^{(k-j)s}\,(1+|k-j|)^{|\alpha|}\,2^{js}\,(1+j)^\alpha\,\left\|u_j\right\|_{L^p}\,.
\end{eqnarray*}
So, under the hypothesis $s>0$, we can conclude thanks to Young's inequality for convolutions.

In the second case $s=0$ and $\alpha\geq0$, it's enough to notice that, as $k\leq j+n_0$,
$$
\frac{1+k}{1+j}\,\leq\,1+n_0\qquad\Longrightarrow\qquad
2^{ks}\,(1+k)^\alpha\,\|\Delta_ku\|_{L^p}\,\leq\,(1+n_0)^\alpha
\sum_{j\geq k-n_0}(1+j)^\alpha\,\left\|u_j\right\|_{L^p}\,.
$$
Taking then the $\ell^\infty$ norm with respect to $k$ gives us the result also in this instance.
\end{proof}

%%%%%%%%%
\subsubsection{Paradifferential calculus in logarithmic classes}
%%%%%%%%%

We now reconsider classical paradifferential calculus results, namely about paraproducts and compositions, in the new logarithmic framework.

Thanks to Littlewood-Paley decomposition, given two tempered distributions $u$ and $v$, formally one can write the product
$u\,v\,=\,\sum_{j,k}\Delta_ju\,\Delta_kv$. Now, due to the
spectral localization of cut-off operators, we have the following \emph{Bony's decomposition} (which was introduced in paper \cite{Bony}):
\begin{equation}\label{eq:bony}
u\,v\,=\,T_uv\,+\,T_vu\,+\,R(u,v)\,,
\end{equation}
where we have defined the paraproduct and remainder operators respectively as
$$
T_uv\,:=\,\sum_jS_{j-2}u\,\Delta_jv\qquad\hbox{ and }\qquad
R(u,v)\,:=\,\sum_j\,\sum_{|k-j|\leq2}\Delta_ju\,\Delta_kv\,.
$$
Let us immediately note that the generic term $S_{j-2}u\,\,\Delta_jv$ is spectrally supported in a dyadic annulus $2^j\wtilde{\mc{C}}$, while,
for all fixed $j$, $\sum_k\Delta_ju\,\Delta_kv$ is spectrally localized in a ball $2^j\mc{B}$. We stress the fact that both
$\wtilde{\mc{C}}$ and $\mc{B}$ are fixed, and they don't depend on $j$.

We start with the continuity properties of the paraproduct operator.

\begin{thm} \label{t:log-pp}
 Let $(s,\alpha,\beta)\,\in\R^3$ and $t>0$. Let also $(p,r,r_1,r_2)$ belong to $[1,+\infty]^4$.

The paraproduct operator 
$T$ maps $L^\infty\times B^{s+\alpha\log}_{p,r}$ in $B^{s+\alpha\log}_{p,r}$,
and  $B^{-t+\beta\log}_{\infty,r_2}\times B^{s+\alpha\log}_{p,r_1}$ in $B^{(s-t)+(\alpha+\beta)\log}_{p,q}$,
with $1/q\,:=\,\min\left\{1\,,\,1/r_1\,+\,1/r_2\right\}$.
Moreover, the following estimates hold:
\begin{eqnarray*}
\|T_uv\|_{B^{s+\alpha\log}_{p,r}} & \leq & C\,\|u\|_{L^\infty}\,\|\nabla v\|_{B^{(s-1)+\alpha\log}_{p,r}} \\
\|T_uv\|_{B^{(s-t)+(\alpha+\beta)\log}_{p,q}} & \leq &
C\,\|u\|_{B^{-t+\beta\log}_{\infty,r_2}}\,\|\nabla v\|_{B^{(s-1)+\alpha\log}_{p,r_1}}\,.
\end{eqnarray*}
Moreover, the second inequality still holds true if $t=0$, when $\beta\leq0$ and $r_2=+\infty$.
\end{thm}

\begin{proof}
As remarked above, the generic term $S_{j-2}u\,\Delta_jv$ is spectrally supported in the ring
$2^j\wtilde{\mc{C}}$, for some fixed ring $\wtilde{\mc{C}}$.
Hence, thanks to Lemma \ref{l:log-B_ind}, it's enough to estimate its $L^p$ norm. 

Applying Lemma \ref{l:log-S_j} gives us the conclusion.
\end{proof}

Let us now state some properties of the remainder operator.

\begin{thm} \label{t:log-r}
 Let $(s,t,\alpha,\beta)\in\R^4$ and $(p_1,p_2,r_1,r_2)\in[1,+\infty]^4$ be such that
$$
\frac{1}{p}\,:=\,\frac{1}{p_1}\,+\,\frac{1}{p_2}\,\leq\,1\qquad\mbox{ and }\qquad
\frac{1}{r}\,:=\,\frac{1}{r_1}\,+\,\frac{1}{r_2}\,\leq\,1\,.
$$
\begin{itemize}
 \item[(i)] If $s+t>0$, then there exists a constant $C>0$ such that, for  any
$(u,v)\in B^{s+\alpha\log}_{p_1,r_1}\times B^{t+\beta\log}_{p_2,r_2}$ we have
$$
\left\|R(u,v)\right\|_{B^{(s+t)+(\alpha+\beta)\log}_{p,r}}\;\leq\;C\,\|u\|_{B^{s+\alpha\log}_{p_1,r_1}}\,
\|v\|_{B^{t+\beta\log}_{p_2,r_2}}\,.
$$
\item[(ii)] If $s+t=0$, $\alpha+\beta\geq0$ and $r=1$, then there exists a $C>0$ such that the inequality
$$
\left\|R(u,v)\right\|_{B^{0+(\alpha+\beta)\log}_{p,\infty}}\;\leq\;C\,\|u\|_{B^{s+\alpha\log}_{p_1,r_1}}\,
\|v\|_{B^{t+\beta\log}_{p_2,r_2}}
$$
holds true for any 
$(u,v)\in B^{s+\alpha\log}_{p_1,r_1}\times B^{t+\beta\log}_{p_2,r_2}$.
\end{itemize}
\end{thm}

\begin{proof}
 We can write $R(u,v)\,=\,\sum_{j}R_j$, where we have set
$$
R_j\,:=\,\sum_{|h-j|\leq2}\Delta_ju\,\Delta_hv\,.
$$
As already pointed out, each $R_j$ is spectrally localized on a ball of radius proportional to
$2^j$. Hence, from Lemma \ref{l:log-ball} and H\"older's inequality we immediately infer the first estimate.

In the second case, we apply the second part of Lemma \ref{l:log-ball}. As a matter of fact, the following
inequality holds true for all $k\geq0$:
$$
(1+k)^{\alpha+\beta}\,\|\Delta_kR(u,v)\|_{L^p}\;\leq\;C\,\sum_{j\geq k-n_0}(1+j)^{\alpha}\,\left\|\Delta_ju\right\|_{L^{p_1}}\,
(1+j)^\beta\,\left\|\Delta_jv\right\|_{L^{p_2}}\,,
$$
where, for simplicity, instead of the full $R_j$, we have considered only the term $\Delta_ju\,\Delta_jv$,
the other ones being similar.

The theorem is completely proved.
\end{proof}

We conclude this part with a result on left composition by smooth functions: it generalizes Proposition 4 of \cite{D2010}
to the logarithmic setting. As the proof is quite technical, we postpone it to the Appendix.

\begin{thm} \label{t:comp}
 Let $I\subset\R$ be an open interval and  $F:\,I\,\longrightarrow\,\R$ a smooth function.
Fix a compact subset $J\subset I$, $(p,r)\in[1,+\infty]^2$ and $(s,\alpha)\in\R^2$ such that $s>0$,
or $s=0$, $\alpha>1$ and $r=+\infty$.

Then there exists a constant $C>0$ such that, for all functions $u$ which are supported in $J$ and with
$\nabla u\in B^{(s-1)+\alpha\log}_{p,r}$, one has $\nabla(F\circ u)\in B^{(s-1)+\alpha\log}_{p,r}$ and
$$
\left\|\nabla\left(F\circ  u\right)\right\|_{B^{(s-1)+\alpha\log}_{p,r}}\,\leq\,C
\left\|\nabla u\right\|_{B^{(s-1)+\alpha\log}_{p,r}}\,.
$$
\end{thm}

\begin{rem} \label{r:comp}
We remark that this statement differs from the classical one (see Chapter 2 of \cite{B-C-D} for instance), as we are looking at the regularity
of the gradient of $F\circ u$, rather than at the regularity of $F\circ u$ itself. In fact, this little difference is more adapted to
our case.

Note also that other extensions (in the same spirit of those in \cite{B-C-D}, Paragraph 2.8.2)
of the previous theorem, under finer assumptions on the function $f$, are possible, but they go beyond the aims of the present paper.
\end{rem}

%%%%%%%%%%%%%%%%%%%%%%%%%%%%%%%%%%%%%%%%%%%%%%%%%%%%%%%%%
\subsection{Zygmund spaces} \label{ss:Zygmund}
%%%%%%%%%%%%%%%%%%%%%%%%%%%%%%%%%%%%%%%%%%%%%%%%%%%%%%%%%

Littlewood-Paley decomposition provides us also with a description of Zygmund and log-Zygmund classes. Before entering into the details, let us recall the ``classical'' case when $p=+\infty$.

\begin{defin} \label{d:LZ}
 A function $g\in L^\infty(\R^d)$ is said to be \emph{log-Zygmund} continuous, and we write $g\in LZ(\R^d)$, if the quantity
$$
|g|_{LZ,\infty}\,:=\,\sup_{z,y\in\R^d,\,0<|y|<1}
\left(\frac{\left|g(z+y)\,+\,g(z-y)\,-\,2\,g(z)\right|}{|y|\,\log\left(1\,+\,\frac{1}{|y|}\right)}\right)\,<\,+\infty\,.
$$
We set $\|g\|_{LZ}\,:=\,\|g\|_{L^\infty}\,+\,|g|_{LZ,\infty}$.

The space $Z(\R^d)$ of \emph{Zygmund} continuous functions is defined instead by the condition
$$
|g|_{Z,\infty}\,:=\,\sup_{z,y\in\R^d,\,0<|y|<1}
\left(\frac{\left|g(z+y)\,+\,g(z-y)\,-\,2\,g(z)\right|}{|y|}\right)\,<\,+\infty\,,
$$
and, analogously, we set $\|g\|_{Z}\,:=\,\|g\|_{L^\infty}\,+\,|g|_{Z,\infty}$.
\end{defin}

More in general, one can define Zygmund classes based on $L^p$ conditions, as follows.

\begin{defin} \label{def:zyg_p}
Let $p\in[1,+\infty]$. We define the space $\mc{Z}_p(\R^d)$ as the set of $f\in L^p(\R^d)$ such that there exists a constant $C>0$ for which
$$
\bigl\|f(\,\cdot\,+y)\,+\,f(\,\cdot\,-y)\,-\,2\,f(\,\cdot\,)\bigr\|_{L^p(\R^d)}\,\leq\,C\,|y|
$$
for all $y\in\R^d$ with $|y|<1$. We denote by $|f|_{\mc{Z}_p}$ the smallest constant $C$ for which the previous inequality is true,
and we set $\|f\|_{\mc{Z}_p}\,:=\,\|f\|_{L^p}+|f|_{\mc{Z}_p}$.

Similarly, the space $\mc{LZ}_p$ is the set of $f\in L^p(\R^d)$ such that, for some constant $C>0$,
$$
\bigl\|f(\,\cdot\,+y)\,+\,f(\,\cdot\,-y)\,-\,2\,f(\,\cdot\,)\bigr\|_{L^p(\R^d)}\,\leq\,C\;|y|\;\log\left(1\,+\,\frac{1}{|y|}\right)
$$
for all $y\in\R^d$ with $|y|<1$. We then set $|f|_{\mc{LZ}_p}$ the smallest constant $C$ for which the previous inequality is true,
and $\|f\|_{\mc{LZ}_p}\,:=\,\|f\|_{L^p}+|f|_{\mc{LZ}_p}$.
\end{defin}

Obviously, for any $p\in[1,+\infty]$, one has $\mc{Z}_p\,\subset\,\mc{LZ}_p$.

Let us recall that $Z\equiv B^1_{\infty,\infty}$ (see e.g. \cite{Ch1995}, Chapter 2, for the proof), while the space $LZ$
coincides with the logarithmic Besov space $B^{1-\log}_{\infty,\infty}$ (see for instance \cite{C-DS-F-M_tl}, Section 3).
Exactly as for the $L^\infty$ instance, the following proposition holds true (for the proof, see e.g. \cite{F-Z}, Section 2).
\begin{prop} \label{p:zygm_p}
For any $p\in[1,+\infty]$, the classes $\mc{Z}_p(\R^d)$ and $\mc{LZ}_p(\R^d)$ coincide, respectively, with the
Besov spaces $B^1_{p,\infty}(\R^d)$ and $B^{1-\log}_{p,\infty}(\R^d)$.
\end{prop}

Then, keeping in mind Theorems \ref{t:log-pp} and \ref{t:log-r} about paraproduct and remainder operators, we immediately infer
the following result.

\begin{coroll} \label{c:z-algebra}
For any $p\in[1,+\infty]$, the spaces $L^\infty\cap\mc{Z}_p$ and $L^\infty\cap\mc{LZ}_p$ are algebras.
\end{coroll}

Let us also recall that, in the classical $L^\infty$ instance, the Zygmund space $Z$ is continuously embedded in the space of
log-Lipschitz functions, and the analogous holds true also for the log-Zygmund class $LZ$ (see \cite{C-DS-F-M_tl},
Section 3).
Next lemma generalize this property to the $L^p$ setting: the proof is analogous to the classical one (see e.g. Proposition
2.107 of \cite{B-C-D}, or Lemma 3.13 of \cite{C-DS-F-M_tl} for the logarithmic case), so we omit it.

\begin{lemma} \label{l:z-ll}
Let $f\in\mc{Z}_p$ or $f\in\mc{LZ}_p$, for some $p\in[1,+\infty]$.

Then, setting $\ell=0$ in the former instance and $\ell=1$ in the latter one, there exists a constant $C>0$ such that, for
all $y\in\R^N$, $|y|<1$, one has
$$
\left\|f(\,\cdot\,+y)\,-\,f(\,\cdot\,)\right\|_{L^p}\,\leq\,C\,|y|\,\log^{1+\ell}\left(1\,+\,\frac{1}{|y|}\right)\,.
$$
\end{lemma}

%%%%
\subsubsection{The case $d=1$}
%%%%

From now on, we will restrict our analysis to the $1$-dimensional case, which is definetely the only one we are interested in.

Thanks to the characterization provided by Proposition \ref{p:zygm_p}, we get the following property.
\begin{coroll} \label{c:z-bound}
Let $d=1$ and fix a $p\in\,]1,+\infty]$. 

Then the spaces $\mc{Z}_p$ and $\mc{LZ}_p$ are algebras continuously embedded in $L^\infty$.
\end{coroll}

\begin{proof}
In view of Corollary \ref{c:z-algebra}, it's enough to prove the embedding in $L^\infty$, of course just for the logarithmic class.

As $1-1/p\,>\,0$, from Proposition \ref{p:log-emb} we infer $B^{1-\log}_{p,\infty}\,\hra\,B^s_{\infty,1}$, for some
$0<s<1-1/p$ (in fact, the loss is logarithmic), and this space is clearly embedded in $L^\infty$.
\end{proof}

\begin{rem} \label{r:log-bound}
Let us stress that the previous statement is not true for $p=1$.
\end{rem}

%%%

Now, given a Zygmund function, we can smooth it out by convolution.
So, let us fix a $p\in[1,+\infty]$ and take a $f\in\mc{Z}_p(\R)$ or $f\in\mc{LZ}_p(\R)$.

Given an even function $\rho\in\mc{C}^\infty_0(\mbb{R})$, $0\leq\rho\leq1$, whose support is contained in the interval $[-1,1]$ and
such that $\int\rho(t)\,dt=1$, we define the mollifier kernel
$$
\rho_\veps(t)\,:=\,\frac{1}{\veps}\,\,\rho\!\left(\frac{t}{\veps}\right)\qquad\qquad\forall\,\veps\in\,]0,1]\,.
$$

Then, for all $\veps\in\,]0,1]$ we set
\begin{equation} \label{eq:f_e}
f_\veps(t)\,:=\,\left(\rho_\veps\,*\,f\right)(t)\,=\,\int_{\mbb{R}_s}\rho_{\veps}(t-s)\,f(s)\,ds\,.
\end{equation}

Let us state some properties about the family of functions we obtain in this way: the following proposition
generalizes the approximation results given in \cite{C-DS-F-M_tl} and \cite{C-DS-F-M_wp} in the $L^\infty$ instance.

\begin{prop} \label{p:z-approx}
Let $p\in[1,+\infty]$ and $f$ belong to $\mc{Z}_p$ or $\mc{LZ}_p$.

Then $\left(f_\veps\right)_\veps$ is a bounded family in $\mc{Z}_p$ or $\mc{LZ}_p$ respectively.

Moreover, there exists a constant $C>0$, depending only on $|f|_{\mc{Z}_p}$ or $|f|_{\mc{LZ}_p}$, such that the following inequalities
hold true for all $\veps\in\,]0,1]$:
\begin{eqnarray}
\left\|f_\veps\,-\,f\right\|_{L^p} & \leq & C\,\,\veps\;\log^\ell\left(1+\frac{1}{\veps}\right)  \label{est:f_e-f} \\
\left\|\d_tf_\veps\right\|_{L^p} & \leq & C\,\log^{1+\ell}\left(1+\frac{1}{\veps}\right) \label{est:d_t-f} \\
\left\|\d^2_tf_\veps(t)\right\|_{L^p} & \leq & C\,\,\frac{1}{\veps}\;\log^\ell\left(1+\frac{1}{\veps}\right) \,, \label{est:d_tt-f}
\end{eqnarray}
where $\ell=0$ or $1$ if $f\in\mc{Z}_p$ or $f\in\mc{LZ}_p$ respectively.
\end{prop}

\begin{proof}
 It's easy to see that \eqref{eq:f_e} can be rewritten as
$$
f_\veps(t)\,=\,\what{\rho}_\veps(D_t)f(t)\,=\,\mc{F}^{-1}_\tau\bigl(\what{\rho}(\veps\tau)\,\what{f}(\tau)\bigr)(t)\,,
$$
where we have set $\what{a}=\mc{F}_t a$ the Fourier transform of $a$ with respect to $t$, $\tau$ the dual variable and
$\mc{F}_\tau^{-1}$ the inverse Fourier transform.

Now, we notice that $\what{\rho}(\veps\tau)$ is a Fourier multiplier, so it commutes with the operators $\Delta_\nu$ of a
Littlewood-Paley decomposition (again, with respect to $t$), and that, for any $\veps\in\,]0,1]$,
$$
\left\|\what{\rho}(\veps\,\cdot\,)\right\|_{L^\infty}\,=\,\left\|\what{\rho}\right\|_{L^\infty}\,\leq\,C\,\|\rho\|_{L^1}\,.
$$
Moreover, as $\rho\in\mc{C}^\infty_0$, then $\what{\rho}\in\mc{S}$ (where $\mc{S}$ denotes the Schwartz class); this implies, in particular,
$$
\left\|\tau^\alpha\,\d^\beta_\tau\what{\rho}(\veps\,\cdot\,)\right\|_{L^\infty}\,\leq\,C_{\alpha,\beta}\,\veps^{|\beta|}\,.
$$
Therefore, $\what{\rho}(\veps\tau)$ is a $S^{0+0\log}$-multiplier (in the sense of Definition \ref{d:log-mult}),
uniformly in $\veps$; then, by Propositions \ref{p:log-mult} and \ref{p:zygm_p}, if $f\in\mc{Z}_p$ we get that also $f_\veps\in\mc{Z}_p$ and
$$
\left\|f_\veps\right\|_{\mc{Z}_p}\,\leq\,C\,\|f\|_{\mc{Z}_p}\,.
$$
The same arguments apply when working in the space $\mc{LZ}_p$.

 Now, let us focus in the case $p<+\infty$, the only new one.

Using the fact that $\rho$ is even and has unitary integral, we can write
$$
f_\veps(t)\,-\,f(t)\,=\,\frac{1}{2\,\veps}\,\int\rho\!\left(\frac{s}{\veps}\right)\bigl(f(t+s)\,+\,f(t-s)\,-\,2f(t)\bigr)\,ds\,.
$$
Then, we take the $L^p$ norm: thanks to Minkowski's inequality (see e.g. \cite{B-C-D}, Chapter 1), we  obtain
$$
\left\|f_\veps\,-\,f\right\|_{L^p}\,\leq\,\frac{1}{2\,\veps}\,\int_{[-\veps,\veps]}\rho(s/\veps)\,
\left\|f(\,\cdot\,+s)\,+\,f(\,\cdot\,-s)\,-\,2f(\,\cdot\,)\right\|_{L^p}\,ds\,.
$$
At this point, estimate \eqref{est:f_e-f} immediately follows, using also the fact that the function
$s\,\mapsto\,s\,\log^\ell\left(1+1/s\right)$ is increasing in $[0,1]$ both for $\ell=0$ and $\ell=1$.

For \eqref{est:d_tt-f} we can argue in the same way, recalling that $\rho''$ is even and that $\int\rho''=0$.

We have to pay attention to the estimate of the first derivative. As $\int\rho'\equiv0$, one has
$$
\d_tf_\veps(t)\,=\,\frac{1}{\veps^2}\,\int_{|s|\leq\veps}\rho'\left(\frac{s}{\veps}\right)\bigl(f(t-s)-f(t)\bigr)ds\,.
$$
Now, we apply Minkowski's inequality, as before, and we use Lemma \ref{l:z-ll}. Estimate \eqref{est:d_t-f} then follows
noticing that the function $s\,\mapsto\,s\,\log^{2}(1+1/s)$ is increasing in $[0,s_0]$, for some $s_0<1$
(and in $[s_0,1]$ it remains strictly positive).

The proposition is now completely proved.
\end{proof}

\begin{rem} \label{r:point}
The previous proposition states that the integral hypothesis on the Zygmund function $f$ gives a control on the corresponding
integral norms of the approximating family $f_\veps$.

Actually, in Corollary \ref{c:z-bound} we have shown that, for $p>1$, the embeddings
$\mc{Z}_p\,\hra\,\mc{LZ}_p\,\hra\,\mc{C}^\sigma$ hold true, for some $0<\sigma<1-1/p$.
Then, arguing as in the last part of the proof,
we immediately infer also pointwise controls: for any $t$,
$$
\bigl|f_\veps(t)\,-\,f(t)\bigr|\,\leq\,C\,\veps^\sigma\,|f|_{\mc{LZ}_p}\qquad\mbox{ and }\qquad
\bigl|\d_tf_\veps(t)\bigr|\,\leq\,\frac{C}{\veps^{1-\sigma}}\,|f|_{\mc{LZ}_p}
$$
(and analogous for the $\mc{Z}_p$ instance). This property obviously extends to matrix-valued functions.
\end{rem}

We conclude this part with fundamental ODE results. The first one mainly states that, in our functional framework, we can solve the equation
$$
f'\,=\,g
$$
only in an approximate way: we can't find an exact primitive of $g$ in the Zygmund classes. The problem is the control of the low
frequencies, and it's very much linked with Bernstein's inequalities.

\begin{prop} \label{p:ode}
\begin{itemize}
\item[(i)] If $f\in B^1_{p,\infty}$ or $B^{1-\log}_{p,\infty}$, then $\d_tf\in B^0_{p,\infty}$ or $B^{0-\log}_{p,\infty}$ respectively.
\item[(ii)] Let $g\in B^0_{p,\infty}$ or $B^{0-\log}_{p,\infty}$.
Then there exists $f\in B^1_{p,\infty}$ or $B^{1-\log}_{p,\infty}$ respectively, such that
$$
r\,:=\,\d_tf\,-\,g\;\in\;B^\infty_{p,\infty}\,:=\,\bigcap_{s\in\R}B^s_{p,\infty}\,.
$$
Moreover, if we have a bounded family $\left(g_\veps\right)_\veps\,\subset\,B^0_{p,\infty}$ or $B^{0-\log}_{p,\infty}$,
then also the corresponding family of solutions $\left(f_\veps\right)_\veps$ and of remainders $\left(r_\veps\right)_\veps$
are bounded sets in their respective functional spaces.
\end{itemize}
\end{prop}

\begin{proof}
The first property is a straightforward consequence of Bernstein's inequalities. So, let us focus on the second statement.

In order to cut off the low frequencies, we fix a function $\theta\in\mc{C}^\infty(\R_\tau)$ such that $\theta\equiv0$ in a neighborhood
of $0$ and $\theta\equiv1$ for $|\tau|\geq1$. We then define $f$ in the phase space by the formula
$$
\what{f}(\tau)\,:=\,\frac{1}{i\,\tau}\,\theta(\tau)\,\what{g}(\tau)\,.
$$

First of all, we want to prove that $f$ belongs to the right space.

We start noticing that, for any $\nu\geq1$, we have $\wtilde{\vphi}_\nu\,:=\,\vphi_{\nu-1}+\vphi_\nu+\vphi_{\nu+1}\,\equiv\,1$
on the support of $\vphi_\nu$ (recall that $\vphi_0\equiv\vphi$, see the beginning of Subsection \ref{ss:L-P}); hence, we infer the equality
\begin{equation} \label{eq:loc}
\vphi_\nu(\tau)\,\what{f}(\tau)\,=\,\frac{1}{i\,\tau}\,\theta(\tau)\,\wtilde{\vphi}_\nu(\tau)\,\vphi_\nu(\tau)\,\what{g}(\tau)\,.
\end{equation}
However, due to the properties of $\theta$, for any $\nu\geq2$ we have $\theta\,\wtilde{\vphi}_\nu\,\equiv\,\wtilde{\vphi}_\nu$. So,
let us define the function $\what{\psi}\in\mc{C}^\infty_0$ (and vanishing near the origin) by the formula
$$
\what{\psi}(\tau)\,:=\,\frac{1}{i\,\tau}\,\wtilde{\vphi}_1(\tau)\,.
$$
Taking the inverse Fourier transform of relation \eqref{eq:loc}, for any $\nu\geq2$ we find
\begin{equation} \label{eq:loc_2}
\Delta_\nu f\,=\,2^{-\nu}\,\psi_\nu\,*\,\Delta_\nu g\,,
\end{equation}
where $\what{\psi}_\nu$ is given by
$\what{\psi}_\nu(\tau)\,:=\,\what{\psi}\left(2^{-\nu}\,\tau\right)$, 
which trivially implies $\psi_\nu(t)\,=\,2^\nu\,\psi\left(2^{\nu}t\right)$.

Let us now consider low frequencies. Again from \eqref{eq:loc} we find
$$
\vphi_1(\tau)\,\what{f}(\tau)\,=\,\frac{1}{i\,\tau}\,\theta(\tau)\,\wtilde{\vphi}_1(\tau)\,\vphi_1(\tau)\,\what{g}(\tau)\quad
\mbox{ and }\quad
\chi(\tau)\,\what{f}(\tau)\,=\,\frac{1}{i\,\tau}\,\theta(\tau)\,\wtilde{\chi}(\tau)\,\chi(\tau)\,\what{g}(\tau)\,,
$$
where we introduced another cut-off function $\wtilde{\chi}$, supported in (say) the ball $B(0,2)$ and
equal to $1$ on the support of $\chi$ (introduced at the beginning of Subsection \ref{ss:L-P}).

From these relations and \eqref{eq:loc_2}, taking into account the properties of $g$, $\theta$, $\wtilde{\vphi}_1$ and $\wtilde{\chi}$,
we immediately get that $f$ belongs to the right Zygmund class.

Now, to prove the regularity of the remainder term, it's enough to observe that
$$
r\,:=\,\d_tf\,-\,g\,=\,\bigl(\theta-1\bigr)\,g\;\in\;B^0_{p,\infty}\qquad\mbox{ (or $B^{0-\log}_{p,\infty}$)}
$$
has compact spectrum, hence it belongs to any $B^s_{p,r}$ for any $s\in\R$ and any $r\in[1,+\infty]$.

Finally, from the previous proof (see also relation \eqref{eq:loc_2} for $f$) one gathers
$$
\|f\|_{B^1_{p,\infty}}\,\leq\,C\,\|g\|_{B^0_{p,\infty}}\qquad\mbox{ and }\qquad
\|r\|_{B^s_{p,\infty}}\,\leq\,C_s\,\|g\|_{B^0_{p,\infty}}\,.
$$
Then, the last sentence immediately follows.
\end{proof}

Note that the previous construction provides us with a linear continuous operator
\begin{equation} \label{eq:def_J}
 J\,:\;B^{0-\ell\log}_{p,\infty}\,\longrightarrow\,B^{1-\ell\log}_{p,\infty}
\end{equation}
(as usual, $\ell=0$ or $\ell=1$ if we are in the Zygmund or in the log-Zygmund instance respectively) such that
$Ju$ is an approximated primitive of $u$:
\begin{equation} \label{eq:J_cont}
\d_tJu\,-\,u\;\in\;B^\infty_{p,\infty}\,.
\end{equation}

The construction of $J$ depends just on the smooth function $\theta$ we fix at the beginning, and it is easy to see that
its norm is given by
$$
\max\left\{\left\|\psi\right\|_{L^1}\;,\;\left\|\mc{F}^{-1}\left(\frac{1}{\tau}\,\theta(\tau)\,\wtilde{\vphi}_1(\tau)\right)\right\|_{L^1}\;,\;
\left\|\mc{F}^{-1}\left(\frac{1}{\tau}\,\theta(\tau)\,\wtilde{\chi}(\tau)\right)\right\|_{L^1}\right\}\,.
$$

For any $\mu\in\N$, let us now set $\theta_\mu(\tau)\,:=\,\theta\bigl(2^{-\mu}\,\tau\bigr)$ and define the
operator $J_\mu$ following the previous construction, but using $\theta_\mu$ instead of $\theta$.
Obviously, properties \eqref{eq:def_J} and \eqref{eq:J_cont} holds true also for $J_\mu$.
\begin{lemma} \label{l:J_mu}
For any $\mu\geq5$, the operator $J_\mu$ maps continuously $B^0_{p,\infty}$ into $B^1_{p,\infty}$, and its norm
$\|J_\mu\|_{\mc{L}(B^0_{p,\infty}\ra B^1_{p,\infty})}$ is independent of $\mu$.

For any fixed $0<s<1$, instead, $\|J_\mu\|_{\mc{L}(B^0_{p,\infty}\ra B^s_{p,\infty})}\,\leq\,C\,2^{-\mu(1-s)}$.

The same holds true in logarithmic classes.
\end{lemma}

\begin{proof}
Arguing as before, for any $\nu\in\N$ we arrive at the formula
$$
\vphi_\nu(\tau)\,\what{f}(\tau)\,=\,\frac{1}{i\,\tau}\,\theta(2^{-\mu}\tau)\,\wtilde{\vphi}_\nu(\tau)\,\vphi_\nu(\tau)\,\what{g}(\tau)\,.
$$

By spectral localization, there exists a $\nu_\mu=\mu+2$ such that, if $\nu>\nu_\mu$ then
$\theta(2^{-\mu}\tau)\,\wtilde{\vphi}_\nu(\tau)\equiv\wtilde{\vphi}_\nu(\tau)$, and there exists
a $\oline{\nu_\mu}=\mu-3$ such that $\theta(2^{-\mu}\tau)\,\wtilde{\vphi}_\nu(\tau)\equiv0$ for $\nu<\oline{\nu_\mu}$.

Hence, for $\nu>\nu_\mu$ formula \eqref{eq:loc_2} holds true, and from it we infer the estimate
\begin{equation} \label{est:J_high}
\|\Delta_\nu f\|_{L^p}\,\leq\,2^{-\nu}\,\|\psi\|_{L^1}\,\|\Delta_\nu g\|_{L^p}\,.
\end{equation}

For $\oline{\nu_\mu}\leq\nu\leq\nu_\mu$, instead, we have the equality
\begin{equation} \label{eq:J_low}
\Delta_\nu f\,=\,2^{-\nu}\,\wtilde{\psi}_{\nu,\nu-\mu}\,*\,\Delta_\nu g\,,
\end{equation}
where we have set $\mc{F}\left(\wtilde{\psi}_{\nu,\nu-\mu}\right)(\tau)\,=\,\mc{F}\left(\wtilde{\psi}_{1,\nu-\mu}\right)(2^{-\nu}\,\tau)$,
with
$$
%\mc{F}\left(\wtilde{\psi}_{\nu,\nu-\mu}\right)(\tau)\,=\,\mc{F}\left(\wtilde{\psi}_{1,\nu-\mu}\right)(2^{-\nu}\,\tau)\,,
%\qquad\mbox{ with }\quad
\mc{F}\left(\wtilde{\psi}_{1,\nu-\mu}\right)(\tau)\,:=\,\frac{1}{i\,\tau}\,\theta(2^{\nu-\mu}\tau)\,\wtilde{\vphi}_1(\tau)\,.
$$
In other words, for any $\nu$, the Fourier transform of $\wtilde{\psi}_{\nu,\nu-\mu}$ is the rescaled of the Fourier
transform of a fixed $\wtilde{\psi}_{1,\nu-\mu}$: actually, this function doesn't depend neither on $\nu$ nor on $\mu$,
but just on their difference $-3\,\leq\,\nu-\mu\,\leq\,2$. Then, the contribution to the norm of the operator $J_\mu$ comes
from the $L^1$ norms of a finite number of terms:
$$
\mc{F}\left(\wtilde{\psi}_{1,-3}\right)\,,\;\mc{F}\left(\wtilde{\psi}_{1,-2}\right)\,\ldots\;
\mc{F}\left(\wtilde{\psi}_{1,2}\right)\,.
$$

%This means that, for any $\nu$, we have to chose
%a different ``base-function'', but we have just a finite number of them, and they just depend on $-3\,\leq\,\nu-\mu\,\leq\,2$
%through the factor $2^{\nu-\mu}$ occurring in the argument of $\theta$.

Therefore, for any $\oline{\nu_\mu}\,\leq\,\nu\,\leq\,\nu_\mu$, by \eqref{eq:J_low} we get
\begin{equation} \label{est:J_low}
 \|\Delta_\nu f\|_{L^p}\,\leq\,2^{-\nu}\,\|\wtilde{\psi}_{1,\nu-\mu}\|_{L^1}\,\|\Delta_\nu g\|_{L^p}\,,
\end{equation}
and the norm of each $\wtilde{\psi}_{1,\nu-\mu}$ doesn't depend on $\nu$, neither on $\mu$.

From \eqref{est:J_high} and \eqref{est:J_low} it's easy to get the conclusion.
\end{proof}

Working component by component, the previous lemma extends to the  case of matrix valued functions.
Hence, we can generalize Proposition \ref{p:ode} to the case of first order systems of ODEs.
\begin{prop} \label{p:ode_system}
Let $k\in\N$ and let $M\,\in\,B^0_{p,\infty}\bigl(\R\,;\,\mc{M}_k(\R)\bigr)$.

Then, for any $X_0\in\R^k$, there exists a vector $X\,\in\,B^1_{p,\infty}\bigl(\R\,;\,\R^k\bigr)$ such that
$X(0)=X_0$ and
$$
\d_tX\,-\,MX\;\in\;B^\infty_{p,\infty}\bigl(\R\,;\;\R^k\bigr)\,.
$$

The same statement holds true in logarithmic Zygmund classes.
\end{prop}

\begin{proof}
By \eqref{eq:J_cont}, it's enough to solve the ``integral'' equation
\begin{equation} \label{eq:ode_int}
X\,=\,J_{\mu}\bigl(MX\bigr)\,+\,X_0
\end{equation}
where $\mu\in\N$ will be chosen later on.

We apply the classical Picard iteration scheme. Let us define $X^0:=X_0$ and, for any $n\in\N$,
$$
X^{n+1}\,:=\,J_\mu\bigl(M\,X^n\bigr)\,+\,X_0\,.
$$
By \eqref{eq:def_J} and Theorems \ref{t:log-pp} and \ref{t:log-r}, it's easy to see that $\bigl(X^n\bigr)_n\,\subset\,B^1_{p,\infty}$.

We claim now that $\bigl(X^n\bigr)_n$ is a Cauchy sequence in the space $B^s_{p,\infty}$, for a fixed
$1/p<s<1$ and for $\mu$ large enough. As a matter of facts, by definition of $X^n$ and Lemma \ref{l:J_mu}, for any $n\in\N$ we have
\begin{eqnarray*}
\left\|X^{n+1}\,-\,X^n\right\|_{B^s_{p,\infty}} & = & \left\|J_\mu\bigl(M\,(X^n\,-\,X^{n-1})\bigr)\right\|_{B^s_{p,\infty}} \\
& \leq & \bigl\|J_\mu\bigr\|_{\mc{L}(B^0_{p,\infty}\ra B^s_{p,\infty})}\,\|M\|_{B^0_{p,\infty}}\,
\left\|X^n\,-\,X^{n-1}\right\|_{B^s_{p,\infty}} \\
& \leq & C\,2^{-\mu(1-s)}\,\|M\|_{B^0_{p,\infty}}\,\left\|X^n\,-\,X^{n-1}\right\|_{B^s_{p,\infty}}
\end{eqnarray*}
and from this estimate we infer the claim.

Therefore, there exists a unique $X\in B^s_{p,\infty}$ such that $X^n\,\ra\,X$ in this space, and then $X$ solves equation
\eqref{eq:ode_int}.
%$$
%X\,=\,J_\mu\bigl(M\,X\bigr)\,+\,X_0\,.
%$$
Using again \eqref{eq:def_J} and the properties of paraproduct and remainder operators, we  see that, actually, $X\in B^1_{p,\infty}$.
Finally, we have that $\d_tX-MX\,=\,\d_tJ_\mu\bigl(MX\bigr)-MX$ belongs to
$B^\infty_{p,\infty}$ by \eqref{eq:J_cont}.
\end{proof}

\begin{rem} \label{r:ode_system}
Notice that, whenever we change $\mu\in\N$, we get a different solution $X^{(\mu)}$ to equation \eqref{eq:ode_int}
with initial datum $X_0$: actually,
by construction these solutions coincide for low enough and high enough frequencies.
However, from Lemma \ref{l:J_mu} we get that their $B^1_{p,\infty}$ norm is independent of $\mu$: then,
thanks to the embedding $B^{s}_{p,\infty}\hra L^\infty$ for a fixed $1/p<s<1$, we can write
\begin{eqnarray*}
\left\|X^{(\mu)}\,-\,X_0\right\|_{L^\infty} & \leq & C\,2^{-\mu(1-s)}\,\|M\|_{B^0_{p,\infty}}\,\left\|X^{(\mu)}\right\|_{B^s_{p,\infty}} \\
& \leq & C\,2^{-\mu(1-s)}\,\|M\|_{B^0_{p,\infty}}\,\left\|X^{(\mu)}\right\|_{B^1_{p,\infty}}\,\leq\,C'\,2^{-\mu(1-s)}\,,
\end{eqnarray*}
where $C'$ doesn't depend on $\mu$.
\end{rem}

%%%%%%%%%%%%%%%%%%%%%%%%%%%%%%%%%%%%%%%%%%%%%%%%%%%%%%%%%
%%%%%%%%%%%%%%%%%%%%%%%%%%%%%%%%%%%%%%%%%%%%%%%%%%%%%%%%%
\section{Proof of the main results} \label{s:proof}
%%%%%%%%%%%%%%%%%%%%%%%%%%%%%%%%%%%%%%%%%%%%%%%%%%%%%%%%%
%%%%%%%%%%%%%%%%%%%%%%%%%%%%%%%%%%%%%%%%%%%%%%%%%%%%%%%%%

Let us now tackle the proof of our main results about energy estimates.

The key will be to build a symmetrizer for our operator $L$. However, in contrast with the classical case, we have to add one step and
continue the construction up to the second order, due to the low regularity of the coefficients $A_j$.

We immediately point out that, up to extend our coefficients out of the interval $[0,T]$,
we can suppose that they are defined on the whole line $\R$. So, the analysis we performed in the pervious section applies.

\medbreak
By hypothesis of hyperbolicity with constant multiplicities, at any point $(t,\xi)\in[0,T]\times\R^n$
we can fix a basis $\bigl(r_j(t,\xi)\bigr)_{1\leq j\leq m}$ of (real-valued normalized)
eigenvectors of $A(t,\xi)$, we can order the eigenvalues in a decreasing way,
$$
\lambda_1(t,\xi)\,\geq\,\ldots\,\geq\,\lambda_m(t,\xi)\,,
$$
and we can write $A(t,\xi)=P(t,\xi)\,\Lambda(t,\xi)\,\left(P^{-1}\right)(t,\xi)$, where we have set
$$
\Lambda(t,\xi)\,:=\,{\rm diag}\bigl(\lambda_1(t,\xi)\,,\,\ldots\,,\,\lambda_m(t,\xi)\bigr)\qquad\mbox{ and }\qquad
P(t,\xi)\,:=\,\bigl(r_1(t,\xi)\,|\,\ldots\,|\,r_m(t,\xi)\bigr)\,.
$$

%Thanks to Remark \ref{r:homog}, let us restrict for a while, until the end of this paragraph,
%to $(t,\xi)\in[0,T]\times\mbb{S}^{n-1}$, where we have defined $\mbb{S}^{n-1}\,:=\,\{\xi\in\R^n\,\bigl|\,|\xi|=1\}$.

\begin{rem} \label{r:smooth}
By linear operators perturbation theory, the hyperbolicity with constant multiplicities implies that the $\lambda_j$'s are analytic functions of the elements of the matrix, and so are the eigenprojectors $\Pi_k$'s.
The eigenvectors, instead, preserve the same regularity of the initial coefficients only locally: in fact, if we choose an eigenbasis $\bigl(r_j(0,\xi)\bigr)_{1\leq j\leq m}$ at $t=0$, then at any $t$ we can extract a basis of eigenvectors from the family $\bigl(\Pi_k(t,\xi)\,r_j(0,\xi)\bigr)_{j,k}$.

One can refer e.g. to
Chapter 2 of \cite{K} or to Appendix 3.I of \cite{Rauch} for a more in-deep analysis of the problem.
\end{rem}

By previous remark, the $\lambda_j$'s and the vectors $r_j$'s inherit (at least locally) the Zygmund regularity of the coefficients of the matrix symbol $A(t,\xi)$. Notice that the local regularity in time is enough for our scopes, due to the local characterization of the Zygmund spaces (recall Definition \ref{def:zyg_p}, and see inequality \eqref{est:d_tE} below).
 As for the dual variable, instead, we will need no special smoothness, as we will
work at any fixed $\xi\in\R^n$.

\medbreak
We now smooth out the coefficients of $\Lambda$ and $P$ by convolution with respect to time, as explained
by formula \eqref{eq:f_e}, and we get the
two matrices $\Lambda_\veps:={\rm diag}\left(\lambda_{1,\veps}\ldots\lambda_{m,\veps}\right)$ and
$P_\veps:=\left(r_{1,\veps}|\ldots|r_{m,\veps}\right)$.

Note that, by the properties of the convolution, we obtain
$$
\lambda_{1,\veps}(t,\xi)\,\geq\,\ldots\,\geq\,\lambda_{m,\veps}(t,\xi)
$$
(where the multiplicities are preserved), while the $r_{j,\veps}$'s are still linearly independent,
at least for small $\veps$ (recall that the set of invertible matrices is open in $\mc{M}_m(\R)$).
Finally, thanks to Proposition \ref{p:z-approx}, the $\lambda_{j,\veps}$'s
and the $r_{j,\veps}$'s have the same Zygmund regularity (with respect to $t$) as $A$, uniformly in $\veps$.

Therefore, if we define, for any $\veps\in\,]0,1]$, any $t\in\R$ and any $\xi\in\R^n$,
\begin{equation} \label{eq:A_eps}
A_\veps(t,\xi)\,:=\,P_\veps(t,\xi)\,\Lambda_\veps(t,\xi)\,\left(P_\veps\right)^{-1}(t,\xi)\,,
\end{equation}
then $A_\veps(t,\xi)$ is still hyperbolic with constant multiplicities (by construction), it preserves the Zygmund regularity with respect to time (by Corollary \ref{c:z-bound}) and it approximates the original matrix $A(t,\xi)$
in the sense of Proposition \ref{p:z-approx}.

Note that, by construction, the line vectors of $Q_\veps:=\left(P_\veps\right)^{-1}$, which we'll call $\,^t\ell_{j,\veps}$, are
left-eigenvectors for $A_\veps$, i.e. they are eigenvectors of the adjoint matrix $A^*_\veps(t,\xi)$:
\begin{equation} \label{eq:l-eigen}
^t\ell_{j,\veps}(t,\xi)\,\cdot\,A_\veps(t,\xi)\,=\,\lambda_{j,\veps}(t,\xi)\,^t\ell_{j,\veps}(t,\xi)
\end{equation}
for all $\veps$, $t$ and $\xi$. Moreover, by definition we have
\begin{equation} \label{eq:l-r}
\ell_{j,\veps}(t,\xi)\,\cdot\,r_{k,\veps}(t,\xi)\,=\,\delta_{jk}\,,
\end{equation}
where we have denoted by $\delta_{jk}$ the Kronecker delta.

%%%%%%%%%%%%%%%%%%%%%%%%%%%%%%%%%%%%%%%%%%%%%%%%%%%%%%%%%
\subsection{Construction of the symmetrizer} \label{ss:symm}
%%%%%%%%%%%%%%%%%%%%%%%%%%%%%%%%%%%%%%%%%%%%%%%%%%%%%%%%%

We present here the key to the proof of the energy estimates: the construction of a symmetrizer for
operator $L$. Actually, for any $\veps$, we will find a symmetrizer for the regularized symbol $A_\veps(t,\xi)$,
defined by relation \eqref{eq:A_eps}, in order to deal with smooth functions, that we can differentiate in
time.

We point out here that we will work at any fixed $\xi\neq0$.

\medbreak
In view of what we said before, we define
\begin{equation} \label{def:S}
S_\veps(t,\xi)\,:=\,S^0_\veps(t,\xi)\,+\,|\xi|^{-1}\,S^1_\veps(t,\xi)\,,
\end{equation}
where $S^0_\veps$ and $S^1_\veps$ are two self-adjoint matrices we have to build up in a suitable way.
We point out here that the role of $S^1_\veps$ is to kill the bad terms coming from the time derivatives of the elements of $S^0_\veps$
in the energy estimates. On the other side, as the second term is of lower order, the time derivatives of $S^1_\veps$ will be easily
controlled in terms of the energy.

We start by proving the following statement.
\begin{lemma} \label{l:symm_01}
Assume the hypothesis of Theorem \ref{th:en_Z} (or Theorem \ref{th:en_LZ}), fix $\xi\neq0$ and define the approximate
matrix symbol $A_\veps$ by relation \eqref{eq:A_eps}.

There exist two families of $m\times m$ real-valued self-adjoint matrices, which are smooth with respect to $t$
and such that:
\begin{itemize}
\item $\bigl(S^0_\veps\bigr)_\veps$ is bounded in $B^1_{p,\infty}$ (or $B^{1-\log}_{p,\infty}$ respectively);
\item the $S^0_\veps$'s are uniformly positive definite: $S^0_\veps v\cdot v\geq C|v|^2$ for any $v\in\C^m$, for a constant $C>0$ depending just on the functional norms of the coefficients of operator $L$;
\item $\bigl(S^1_\veps\bigr)_\veps$ is bounded in $B^0_{p,\infty}$ (or $B^{0-\log}_{p,\infty}$ respectively).
\end{itemize}
Moreover, for any $\veps\in\,]0,1]$, the matrices $S^0_\veps$ and $S^1_\veps$ satisfy the relation
\begin{equation} \label{eq:symm_1}
\d_tS^0_\veps\what{u}\,\cdot\,\what{u}\,+\,2\,\Re\left(-i\,|\xi|^{-1}\,S^1_\veps A_\veps\what{u}\,\cdot\,\what{u}\right)\,=\,R_\veps\what{u}\,\cdot\,\what{u}\,,
\end{equation}
where the family of remainders $\bigl(R_\veps\bigr)_\veps$ is bounded in $L^\infty\bigl([0,T];\mc{M}_m(\R)\bigr)$.
\end{lemma}

\begin{proof}
For notation convenience, from now on we will drop out the dependence on $\veps$, even if we will
always work with smoothed matrices, vectors and coefficients.

Let us write $S^0$ and $S^1$ in the form (recall that we have defined $Q=P^{-1}$)
$$
S^0\,=\,Q^*\,\Sigma^0\,Q\qquad\mbox{ and }\qquad
S^1\,=\,Q^*\,\Sigma^1\,Q\,,
$$
where $\Sigma^0$ and $\Sigma^1$ are two suitable self-adjoint matrices, to be found.
In particular, we will construct them such that $\Sigma^0$ is real-valued,
while $\Sigma^1$ is pure imaginary.

Notice that, for any $v\in\C^m$,
\begin{eqnarray*}
2\Re\left(-i\,|\xi|^{-1}\,S^1 Av\,\cdot\,v\right) & = &
2\Re\left(-i\,|\xi|^{-1}\,Q^*\,\Sigma^1\,\Lambda\,Qv\,\cdot\,v\right) \\
& = & |\xi|^{-1}\left(-iQ^*\Sigma^1\Lambda Qv\cdot v\,+\,\oline{-iQ^*\Sigma^1\Lambda
Qv\cdot v}\right) \\
& = & i\,|\xi|^{-1}\left(-Q^*\Sigma^1\Lambda Qv\cdot v\,+\,Q^*\Lambda\Sigma^1 Qv\cdot v\right)\,,
\end{eqnarray*}
where, in the last step, we passed to the adjoint and we used the properties of the scalar product.

Using this last relation, we can rewrite the left-hand side of equation \eqref{eq:symm_1} in the form
$$
\d_tS^0\what{u}\,\cdot\,\what{u}\,+\,2\,\Re\left(-i\,|\xi|^{-1}\,S^1 A\what{u}\,\cdot\,\what{u}\right)\,=\,
Q^*\,G\,Q\,\what{u}\,\cdot\,\what{u}\,,
$$
where, setting $\Theta\,:=\,\d_tQ\,Q^{-1}\,=\,\d_tQ\,P$ and $[A,B]=AB-BA$ the commutator between two operators, we have defined
\begin{equation} \label{eq:G}
G\,:=\,\d_t\Sigma^0\,+\,\Sigma^0\,\Theta\,+\,\Theta^*\,\Sigma^0\,+\,i\,|\xi|^{-1}\,\left[\Lambda,\Sigma^1\right]\,.
\end{equation}

Therefore, we are going to construct $\Sigma^0$ and $\Sigma^1$ in order to satisfy the relation $G=0$ in an approximate
way.

%%%
\paragraph{Strictly hyperbolic case.}
Let us consider for a while the  strictly hyperbolic case, i.e. all the 
$\lambda_j$'s are distinct. In this case, we impose $\Sigma^0$ diagonal.

We will proceed in two steps: first of all, we will use $\Sigma^0$ to cancel out the diagonal terms 
of $G$, and then $\Sigma^1$ to put also the other elements to $0$.

Before going on, some notations are in order. We set
$$
\Sigma^0\,=\,{\rm diag}\bigl(\sigma_j\bigr)_{1\leq j\leq m}\qquad\mbox{ and }\qquad
\Sigma^1\,=\,\bigl(\wtilde{\sigma}_{jk}\bigr)_{1\leq j,k\leq m}\,;
$$
recall that, by our requirements, $\Sigma^1$ is null on the diagonal: $\wtilde{\sigma}_{jj}=0$. We denote also
$G\,=\,\bigl(g_{jk}\bigr)_{j,k}$ and $\Theta\,=\,\bigl(\theta_{jk}\bigr)_{j,k}$. Note that all these matrices, except $\Sigma^1$,
are real-valued.

\smallbreak
\textit{(i) $G\sim0$: diagonal elements}

Let us read equation \eqref{eq:G} on the diagonal terms: for any $1\leq j\leq m$ we find
$$
g_{jj}\,=\,\d_t\sigma_j\,+\,2\,\sigma_j\,\theta_{jj}\,=\,0\,.
$$
In order to solve this equation, we set $\sigma_j(t)=\exp\bigl(\omega_j(t)\bigr)$, with $\sigma_j(0)=1$: then,
recalling also the definition of $\Theta$, we are reconducted to the ODEs (for any $j$)
\begin{equation} \label{eq:ode_omega}
\d_t\omega_j\,=\,-2\,\theta_{jj}\,=\,-2\,\d_t\ell_j\,\cdot\,r_j\,=\,
-2\sum_{h=1}^m\d_tq_{jh}\,p_{hj}\,,
\end{equation}
with the initial condition $\omega_j(0)=0$.

As already pointed out, the $q_{jk}$'s and $p_{jk}$'s both belong to $B^1_{p,\infty}$ (under hypothesis \eqref{hyp:Z}) or
$B^{1-\log}_{p,\infty}$ (under hypothesis \eqref{hyp:LZ}).
So, thanks to the embedding $B^0_{p,\infty}\hra B^{-1/p}_{\infty,\infty}$ (or the analogous one in the logarithmic instance) and to
Theorems \ref{t:log-pp} and \ref{t:log-r} (recall also that $p>1$ by assumption), it's easy to see that the right-hand side of
\eqref{eq:ode_omega} belongs to $B^0_{p,\infty}$ (or $B^{0-\log}_{p,\infty}$ respectively).

Then, by Proposition \ref{p:ode}, we can find an approximate solution of \eqref{eq:ode_omega}: there exist functions
$\omega_j\in B^1_{p,\infty}$ (or $ B^{1-\log}_{p,\infty}$ rispectively) and $\rho_j\in B^\infty_{r,\infty}$ such that
$$
\d_t\omega_j\,+\,2\,\theta_{jj}\,=\,\rho_j\,.
$$
So, the matrix $\Sigma^0$ is determined. Notice that there exists a constant $K$ such that $\Sigma^0\,\geq\,K\Id$.

\smallbreak
\textit{(ii) $G\sim0$: terms out of the diagonal}

Let us now consider equation \eqref{eq:G} out of the diagonal. For any $j\neq k$ we easily get
\begin{equation} \label{eq:g_jk}
g_{jk}\,=\,\sigma_j\,\theta_{jk}\,+\,\sigma_k\,\theta_{kj}\,+\,i\,|\xi|^{-1}\,\left(\lambda_j\,-\,\lambda_k\right)\,\wtilde{\sigma}_{jk}\,.
\end{equation}

We want a more explicit formula. Differentiating relation \eqref{eq:l-eigen} with respect to time, we find
$$
^t\d_t\ell_j\,\left(A\,-\,\lambda_j\right)\,+\,^t\ell_j\,\left(\d_tA\,-\,\d_t\lambda_j\right)\,=\,0\,.
$$
We now evaluate it on the right eigenvector $r_k$, with $k\neq j$: keeping in mind \eqref{eq:l-r}, we arrive to the following expression:
\begin{equation} \label{eq:theta}
\left(\lambda_k\,-\,\lambda_j\right)\,\d_t\ell_j\,\cdot\,r_k\,+\,\ell_j\,\d_tA\,r_k\,=\,0\qquad\Longrightarrow\qquad
\theta_{jk}\,=\,\frac{1}{\lambda_j-\lambda_k}\,\ell_j\,\d_tA\,r_k\,.
\end{equation}
Putting this formula into \eqref{eq:g_jk}, one immediately gathers, for any $j\neq k$,
\begin{equation} \label{eq:Sigma_1}
\wtilde{\sigma}_{jk}\,=\,\frac{i\,|\xi|}{\left(\lambda_j-\lambda_k\right)^2}\,\bigl(\sigma_j\,\ell_j\,\d_tA\,r_k\,-\,
\sigma_k\,\ell_k\,\d_tA\,r_j\bigr)\,.
\end{equation}
Note that, by Remark \ref{r:homog}, each $\wtilde{\sigma}_{jk}$ is a homogeneous function of degree $0$ in $\xi$.

The symmetrizer is now completely determined in the case of strict hyperbolicity.

\paragraph{The case of multiplicities bigger than $1$.}
Recall that we denoted by $m_h$ the multiplicity of the eigenvalue $\lambda_h$: then,
the matrix $\Lambda$ is block diagonal, and each block is the $m_h\times m_h$ matrix $\lambda_h\Id$, i.e.
$$
\lambda_j\,=\,\lambda_h\qquad\qquad\mbox{ for all }\qquad j_h\,\leq\,j\,\leq\,j_h+m_h-1\,,
$$
where $j_h=\sum_{l=1}^{h-1}m_l$.

So, we will define $\Sigma^0$ to be zero out of the diagonal blocks: $\Sigma^0\,:=\,\bigl(\sigma_{jk}\bigr)_{j,k}$, with
$$
(j,k)\,\not\in\,[j_h,j_h+m_h-1]^2\quad\mbox{for all }\,h\qquad\Longrightarrow\qquad\sigma_{jk}\,=\,0\,.
$$
$\Sigma^1=\bigl(\wtilde{\sigma}_{jk}\bigr)_{j,k}$, instead, will be taken with all $0$'s in these blocks.

As done before, let us first annihilate the diagonal blocks of $G$ (up to a remainder). Take $j$: for some
$h$, $j\in[j_h,j_h+m_h-1]$. Then, for $k$ in the same interval, we will have
\begin{equation} \label{eq:g_jk-mult}
g_{jk}\,=\,\d_t\sigma_{jk}\,+\,\sum_{l=j_h}^{j_h+m_h-1}\bigl(\sigma_{jl}\,\theta_{lk}\,+\,\theta_{lj}\,\sigma_{kl}\bigr)\,.
\end{equation}
It's possible to see that the condition $g_{jk}=0$, for such a range of $j$ and $k$, is equivalent to a system
of first order ODEs
\begin{equation} \label{eq:ode_mult}
\d_tX\,=\,MX\,,
\end{equation}
where the vector $X$ contains the $\sigma_{jk}$'s, while the coefficients of the matrix $M$ are sums of elements of $\Theta$.
Hence, $M\in B^0_{p,\infty}$ (or $B^{0-\log}_{p,\infty}$ in the logarithmic instance). Then, applying Proposition \ref{p:ode_system},
we can solve \eqref{eq:ode_mult} with the initial condition $\sigma_{jk}(0)=\delta_{jk}$: there exist
$\bigl(\sigma_{jk}\bigr)_{jk}\,\subset\,B^{1}_{p,\infty}$ (or $B^{1-\log}_{p,\infty}$) and 
$\bigl(\rho_{jk}\bigr)_{jk}\,\subset\,B^\infty_{p,\infty}$ such that
$$
\d_t\sigma_{jk}\,+\,\sum_{l=j_h}^{j_h+m_h-1}\bigl(\sigma_{jl}\,\theta_{lk}\,+\,\theta_{lj}\,\sigma_{kl}\bigr)\,=\,\rho_{jk}\,.
$$

Notice that, by equation \eqref{eq:g_jk-mult}, it's easy to see that $\Sigma^0$ is symmetric. Moreover, by Remark \ref{r:ode_system},
we can solve \eqref{eq:ode_mult} for $\mu$ large enough, such that $\Sigma^0$ remains a strictly positive matrix
for all $(t,\xi)$: for any $v\in\C^m$,
$$
\Sigma^0v\,\cdot\,v\,\geq\,|v|^2\,-\,\left|\left(\Sigma^0-\Id\right)v\,\cdot\,v\right|\,\geq\,\left(1\,-\,C'\,2^{-\mu(1-s)}\right)|v|^2\,,
$$
where $1/p<s<1$ is the index we fixed in Proposition \ref{p:ode_system}.

Let us now work on the terms out of the diagonal blocks. Fix $j\in[j_h,j_h+m_h-1]$ and $k\in[j_{h'},j_{h'}+m_{h'}-1]$, with $h\neq h'$.
Then \eqref{eq:g_jk} becomes
$$
g_{jk}\,=\,\sum_{l=j_{h'}}^{j_{h'}+m_{h'}-1}\theta_{lj}\,\sigma_{lk}\,+\,\sum_{l=j_h}^{j_h+m_h-1}\sigma_{jl}\,\theta_{lk}\,+\,
i\,|\xi|^{-1} \wtilde{\sigma}_{jk}\,\bigl(\lambda_{h'}\,-\,\lambda_h\bigr)\,.
$$
If we replace now the values of $\theta_{jk}$ given by \eqref{eq:theta}, it's easy to see that a formula like \eqref{eq:Sigma_1}
still holds true. In particular, $\Sigma^1$ is self-adjoint and homogeneous of degree $0$ in $\xi\neq0$.

\medbreak
So, let us sum up what we have found.

Thanks to the previous computations, we constructed approximated matrices $S_\veps$ of the form \eqref{def:S}. Notice that
$S^0_\veps$ and $S^1_\veps$ fulfill relation \eqref{eq:symm_1}, with the remainder defined by the matrix
$R_\veps\,:=\,{\rm diag}\left(\rho_{j,\veps}\exp\bigl(\omega_{j,\veps}\bigr)\right)_{1\leq j\leq m}$ in the strictly
hyperbolic case, and $R_\veps\,:=\,\bigl(\rho_{jk,\veps}\bigr)_{j,k}$ in the one with constant multiplicities.

We point out also that the family $\left(S^0_\veps\right)_\veps$ is bounded in $B^1_{p,\infty}$
(or $B^{1-\log}_{p,\infty}$), $\left(S^1_\veps\right)_\veps$ is bounded in $B^0_{p,\infty}$ (or $B^{0-\log}_{p,\infty}$ respectively) and, thanks to Theorem \ref{t:comp} and embeddings,
$\left(R_\veps\right)_\veps$ is bounded in the space $L^\infty$.

Lemma \ref{l:symm_01} is completely proved.
\end{proof}

We now link the approximation parameter with the dual variable, following the original idea of paper \cite{C-DG-S}: we set
\begin{equation} \label{eq:approx-fourier}
\veps\,=\,|\xi|^{-1}\,.
\end{equation}
Note that, in this way, we will restrict to the case of high frequencies, more precisely to $|\xi|\geq1$. However, for low
frequencies it's easy to get the desired estimates (see the next subsection).

Then, the matrix symbol $S_{1/|\xi|}$, defined by \eqref{def:S},
is a \emph{microlocal symmetrizer} for the approximated system
$$
L_\veps u\,=\,\d_tu\,+\,\sum_{j=1}^nA_{j,\veps}(t)\,\d_ju\,.
$$
More precisely, we have the following proposition.
\begin{prop} \label{p:symm}
Let us define $S_\veps$ by equation \eqref{def:S}, with $S^0_\veps$ and $S^1_\veps$ given by Lemma \ref{l:symm_01}
and $\veps$ given by the choice \eqref{eq:approx-fourier}.

Then, $S_{1/|\xi|}$ enjoys the following properties:
\begin{itemize}
\item $S_{1/|\xi|}$ is self-adjoint;
\item for any $t$ and $|\xi|\geq R_0$ (for a $R_0>0$ just depending on the constants $K_z$ in \eqref{hyp:Z} and $K_{\ell z}$ in
\eqref{hyp:LZ} for the Zygmund and log-Zygmund instances respectively),
it's self-adjoint, uniformly bounded and (uniformly) positive definite:
there exist constants $0<K_1\leq K_2$ such that, for any $v\in\C^m$,
$$
K_1\,|v|^2\;\leq\;S_{1/|\xi|}v\,\cdot\,v\;\leq\;K_2\,|v|^2\,;
$$
\item for all $(t,\xi)$, the matrix $S^0_{1/|\xi|}(t,\xi)\,A_{1/|\xi|}(t,\xi)$ is self-adjoint.
\end{itemize}
\end{prop}

\begin{proof}
The proposition is an immediate consequence of Lemma \ref{l:symm_01}.

Notice that, in the lower bound $K_1|v|^2\,\leq\,S_{1/|\xi|}v\cdot v$, we used Remark \ref{r:point} to control from below the time derivatives which appear in $S^1$.
\end{proof}

Notice that the present notion of microlocal symmetrizability differs from the one of Definition \ref{d:micro_symm}
in the regularity we require with respect to time (here, we don't have uniform Lipschitz continuity) and in the requirement that $SA$ is self-adjoint (this is true, in our case, just for the highest order part $S^0$).

%%%%%%%%%%%%%%%%%%%%%%%%%%%%%%%%%%%%%%%%%%%%%%%%%%%%%%%%%
\subsection{Energy estimates}
%%%%%%%%%%%%%%%%%%%%%%%%%%%%%%%%%%%%%%%%%%%%%%%%%%%%%%%%%

We are now ready to prove the energy estimates.

First of all, by Fourier transform we pass to the phase space, where system \eqref{def:Lu} reads
\begin{equation} \label{eq:fourier-syst}
\what{Lu}(t,\xi)\,=\,\d_t\what{u}(t,\xi)\,+\,i\,A(t,\xi)\cdot\what{u}(t,\xi)\,.
\end{equation}

So, we define the approximate energy in the Fourier variable:
\begin{equation} \label{def:en}
E_\veps(t,\xi)\,:=\,S_\veps(t,\xi)\what{u}(t,\xi)\,\cdot\,\what{u}(t,\xi)\,,
\end{equation}
where the approximated symmetrizer $S_\veps$ is given by \eqref{def:S}.

Recall that we have fixed $\veps=1/|\xi|$ in \eqref{eq:approx-fourier}. Nevertheless, for convenience we will keep, for the moment,
the notation with $\veps$.

Recall also that we will work with $|\xi|\geq R_0$. However, in the case $|\xi|\leq R_0$ energy estimates immediately follow:
it's enough to take the scalar product (in $\R^m$) of equation \eqref{eq:fourier-syst} by $\what{u}$, to use the bound \eqref{hyp:bound}
for the $A_j$'s (thanks to Corollary \ref{c:z-bound}), and to apply Gronwall's lemma after an integration in time.

So, let us come back to the energy $E_\veps$. Due to Proposition \ref{p:symm}, Corollary \ref{c:z-bound} and taking into account again Remark \ref{r:point}, it's easy to see that, for any $t$ and $|\xi|\geq R_0$,
$$
E_\veps(t,\xi)\,\sim\,|\what{u}(t,\xi)|^2\,.
$$

Now, we differentiate the energy with respect to time: using also the fact that $S_\veps$ is self-adjoint, it's easy to get the equality
$$
\d_tE_\veps\,=\,\d_tS_\veps\what{u}\,\cdot\,\what{u}\,+\,2\,\Re\left(S_\veps\d_t\what{u}\,\cdot\,\what{u}\right)\,.
$$
We then exploit equation \eqref{eq:fourier-syst} and definition \eqref{def:S}, and we arrive to
\begin{eqnarray*}
\d_tE_\veps & = & \d_tS^0_\veps\what{u}\,\cdot\,\what{u}\,+\,|\xi|^{-1}\,\d_tS^1_\veps\what{u}\,\cdot\,\what{u}\,+\,
2\,\Re\left(-\,i\,S_\veps(A-A_\veps)\what{u}\,\cdot\,\what{u}\right)\,+ \\
& & +\,2\,\Re\left(S_\veps\what{Lu}\,\cdot\,\what{u}\right)\,+\,
2\,\Re\left(-i\,S^0_\veps A_\veps\what{u}\,\cdot\,\what{u}\right)\,+\,
2\,\Re\left(-i\,|\xi|^{-1}\,S^1_\veps A_\veps\what{u}\,\cdot\,\what{u}\right)
\end{eqnarray*}
By construction, $S^0_\veps$ and $S^1_\veps$ satisfy relation \eqref{eq:symm_1}. Moreover,
$S^0_\veps\,A_\veps$ is self-adjoint; so, for any $v\in\C^m$, the
quantity $S^0_\veps A_\veps v\cdot v$ belongs to $\R$. Hence, the previous relation becomes
\begin{equation} \label{eq:E_t}
\d_tE_\veps\,=\,2\Re\left(S_\veps\what{Lu}\,\cdot\,\what{u}\right)\,+\,
2\Re\bigl(-\,i\,S_\veps(A-A_\veps)\what{u}\,\cdot\,\what{u}\bigr)\,+\,
R_\veps\,\what{u}\,\cdot\,\what{u}\,+\,
|\xi|^{-1}\d_tS^1_\veps\what{u}\,\cdot\,\what{u}\,.
\end{equation}

Let us deal with the first term. As $S_\veps$ is self-adjoint and positive definite, it defines a scalar product, for
which Cauchy-Schwarz inequality applies:
\begin{eqnarray*}
\left|2\,\Re\left(S_\veps\what{Lu}\,\cdot\,\what{u}\right)\right| & \leq & \left(S_\veps\what{Lu}\,\cdot\,\what{Lu}\right)^{1/2}\;
\left(S_\veps\what{u}\,\cdot\,\what{u}\right)^{1/2} \\
& \leq & C\,\left|\what{Lu}(t,\xi)\right|\,\bigl(E_\veps(t,\xi)\bigr)^{1/2}\,.
\end{eqnarray*}

Using again boundedness and positivity of $S_\veps$, the second term can be bounded as follows:
$$
\left|2\Re\left(-\,i\,S_\veps(A-A_\veps)\what{u}\,\cdot\,\what{u}\right)\right|\,\leq\,C\,\bigl|A-A_\veps\bigr|_{\mc{M}}\,E_\veps(t,\xi)\,.
$$
Recall that both $A$ and $A_\veps$ are homogeneous of degree $1$ in $\xi$.

As pointed out before, the coefficients of the third term are bounded, both under the Zygmund and the log-Zygmund hypothesis. Then we find
$$
\left|R_\veps\,\what{u}\,\cdot\,\what{u}\right|\,\leq\,C\,E_\veps(t,\xi)\,.
$$

Finally, for the last element of \eqref{eq:E_t} we just use the properties of $S_\veps$ to write
$$
\left|\d_tS^1_\veps\what{u}\,\cdot\,\what{u}\right|\,\leq\,C\,\left|\d_tS^1_\veps\right|_{\mc{M}}\,\left|\what{u}\right|^2\,\leq\,
C\,\left|\d_tS^1_\veps\right|_{\mc{M}}\,E_\veps\,.
$$

Therefore, by \eqref{eq:E_t} and the previous bounds we infer the inequality
\begin{equation} \label{est:d_tE}
\d_tE_\veps(t,\xi)\,\leq\,C\left(\left|\what{Lu}(t,\xi)\right|\,\bigl(E_\veps(t,\xi)\bigr)^{1/2}+
\left(1+|A-A_\veps|_{\mc{M}}+|\xi|^{-1}\left|\d_tS^1_\veps(t,\xi)\right|_{\mc{M}}\right)E_\veps(t,\xi)\right).
\end{equation}
Starting from this relation, if we define $e_\veps(t,\xi):=\bigl(E_\veps(t,\xi)\bigr)^{1/2}$ and
$$
\Phi(t,\xi)\,:=\,C\left(t+\int^t_0\left(|\xi|^{-1}\left|\d_tS^1_\veps(\tau,\xi)\right|_{\mc{M}}+
|A(\tau,\xi)-A_\veps(\tau,\xi)|_{\mc{M}}\right)d\tau\right),
$$
Gronwall's inequality immediately entails, for any $t\in[0,T]$,
\begin{equation} \label{est:en_part}
e_\veps(t,\xi)\,\leq\,e^{\int^t_0\Phi(\tau,\xi)d\tau}\,e_\veps(0,\xi)+\int^t_0e^{\int^t_\tau\Phi(s,\xi)ds}\,
\left|\what{Lu}(\tau,\xi)\right|\,d\tau\,.
\end{equation}

Now we set $\g=1/p'$ and we apply H\"older's inequality to the time integral in the exponential term: hence we find
\begin{eqnarray*}
|\xi|^{-1}\int^t_0\left|\d_tS^1_\veps(\tau,\xi)\right|_{\mc{M}}d\tau & \leq & C\,|\xi|^{-1}\,t^\g\,
\left\|\d_tS^1_\veps(\,\cdot\,,\xi)\right\|_{L^p([0,T];\mc{M}_m(\R))} \\
\int^t_0\left|A-A_\veps\right|_{\mc{M}}\,d\tau & \leq & C\,|\xi|\,t^\g\,
\sup_{1\leq j\leq n}\|A_j-A_{j,\veps}\|_{L^p([0,T];\mc{M}_m(\R))}\,.
\end{eqnarray*}
Recalling the definition of $S^1_\veps$, we can see that, in total, we have two time derivatives: they can act on two different terms,
or on the same. In any case, we can apply Proposition \ref{p:z-approx} (with $\veps=1/|\xi|$), which leads us to
$$
|\xi|^{-1}\int^t_0\left|\d_tS^1_\veps(\tau,\xi)\right|_{\mc{M}}d\tau\,\leq\,C\,t^\g\,
\log^\ell\left(1+|\xi|\right)\,,
$$
where, as usual, $\ell=0$ under the $\mc{Z}_p$ hypothesis and $\ell=1$ in the $\mc{LZ}_p$ case.
Again by Proposition \ref{p:z-approx}, we get also
$$
\int^t_0\left|A-A_\veps\right|_{\mc{M}}\,d\tau\,\leq\,C\,t^\g\,\log^\ell\left(1+|\xi|\right)\,.
$$

Putting this control into \eqref{est:en_part}, we finally get, for all $t\in[0,T]$ and all $|\xi|\geq1$,
$$
\left|\what{u}(t,\xi)\right|\,\leq\,C\,e^{C\,t}\left(|\xi|^{\ell\,\wtilde{\beta}\,t^\g}\,\left|\what{u}(0,\xi)\right|\,+\,
\int^t_0|\xi|^{\ell\,\wtilde{\beta}\,(t-\tau)^\g}\left|\what{Lu}(\tau,\xi)\right|\,d\tau\right)\,,
$$
for some suitable positive constants $C$ and $\wtilde{\beta}$. Taking the $L^2$ norm (or the $H^s$ norm, for any $s$) with respect to
$\xi$ completes the proof.

\begin{rem} \label{r:loss}
As noticed in Remark \ref{r:th_loss}, an inequality like \eqref{est:en_part} would be suitable for iterations in time. However,
we have to make the $L^p$ norm appear to control the behaviour of the coefficients, and this gives the strange factor $t^\g$.
Then, for log-Zygmund coefficients we cannot improve the inequality in Theorem \ref{th:en_LZ}.

Note also that this is not the case for Zygmund coefficients, or for $p=+\infty$.
\end{rem}

\begin{rem} \label{r:p=1}
We remarked several times the fact that $p$ has to be bigger than $1$. We used the condition $p>1$ in these occasions:
\begin{itemize}
 \item for the embeddings $\mc{Z}_p\,\hookrightarrow\,L^\infty$ and $\mc{LZ}_p\,\hookrightarrow\,L^\infty$, and actually in (pointwise)
H\"older classes (recall also Remark \ref{r:point});
\item to have $fg\,\in\,B^{0-\log}_{p,\infty}$ for $f\,\in\,B^{1-\log}_{p,\infty}$ and $g\,\in\,B^{0-\log}_{p,\infty}$:
this property is not true, in general, if $p=1$ (recall Theorems \ref{t:log-pp} and \ref{t:log-r}).
\end{itemize}
In particular, $p>1$ is fundamental to recover the equivalence between our energy and the classical one, and to construct
the symmetrizer solving the corresponding ODEs.
\end{rem}

%%%%%%%%%%%%%%%%%%%%%%%%%%%%%%%%%%%%%%%%%%%%%%%%%%%%%%%%%
%%%%%%%%%%%%%%%%%%%%%%%%%%%%%%%%%%%%%%%%%%%%%%%%%%%%%%%%%
\section{An application: the case of the wave equation} \label{s:wave}
%%%%%%%%%%%%%%%%%%%%%%%%%%%%%%%%%%%%%%%%%%%%%%%%%%%%%%%%%
%%%%%%%%%%%%%%%%%%%%%%%%%%%%%%%%%%%%%%%%%%%%%%%%%%%%%%%%%

As an application of the previous results, let us consider the case of the second order scalar equations, for simplicity in the case
of space dimension $n=1$.

We will also show that, in this instance, the restriction $p>1$ is not necessary: energy estimates, with or without loss, hold true also
for $p=1$. This is in accordance with the results in \cite{Tar}.

\medbreak
So, let $\alpha(t)\in L^\infty$, $0<\alpha_*\leq\alpha\leq\alpha^*$ and suppose that $\alpha\in B^1_{p,\infty}$ or $B^{1-\log}_{p,\infty}$,
for some $p\in[1,+\infty]$. We consider the wave operator
\begin{equation} \label{eq:wave}
Wu(t,x)\,:=\,\d_t^2u(t,x)\,-\,\alpha(t)\,\d^2_xu(t,x)\,.
\end{equation}

If we now set
$$
U(t,x)\,:=\,\left(\begin{array}{c}
-\d_xu \\[1ex] \d_tu
\end{array}\right)\,,\qquad
LU(t,x)\,:=\,\left(\begin{array}{c}
0 \\[1ex] Wu
\end{array}\right)\qquad\mbox{ and }\qquad
A(t)\,:=\,\left(\begin{array}{cc}
0 & 1 \\[1ex] \alpha(t) & 0
\end{array}\right)\,,
$$
then \eqref{eq:wave} is equivalent to the first order system
$$
LU(t,x)\,=\,\d_tU(t,x)\,+\,A(t)\,\d_xU(t,x)\,.
$$

For convenience, let us assume $\alpha$ to be smooth, and forget about the convolution and the approximation index $\veps$ in the notations.
We also set $\alpha=a^2$, with $a\geq a_*>0$.

An easy computation shows us that the eigenvalues and respective (normalized) eigenvectors of the matrix $A(t,\xi)=\xi A(t)$ are
$$
\lambda_{\pm}(t,\xi)\,:=\,\pm\,a(t)\,\xi\qquad\mbox{ and }\qquad
r_{\pm}(t,\xi)\,:=\,\bigl(1+a^2\bigr)^{-1/2}\,\bigl(1\,,\, \pm a\bigr)\,.
$$
Therefore, choosing $\lambda_+$ as the first eigenvalue, we find the matrices of change of basis:
$$
P\,:=\,\frac{1}{\sqrt{1+a^2}}\left(\begin{array}{cc}
1 & 1 \\[1ex] a(t) & -a(t)
\end{array}\right)\qquad\mbox{ and }\qquad
Q\,:=\,P^{-1}\,=\,\frac{\sqrt{1+a^2}}{2}\left(\begin{array}{cc}
1 & 1/a(t) \\[1ex] 1 & -1/a(t)
\end{array}\right)\,.
$$
We also set $\ell_{\pm}\,:=\,\bigl(\sqrt{1+a^2}/2\bigr)\,\bigl(1\,,\,\pm1/a\bigr)$. Note that, as already pointed out in the general
computations, they are indeed left eigenvectors of the matrix $A$.

We now perform straightforward computations, and we get
$$
\d_t\ell_\pm\,=\,\frac{a\,\d_ta}{2\,\sqrt{1+a^2}}\,\left(1\,,\,\pm\,\frac{1}{a}\right)\,+\,\frac{\sqrt{1+a^2}}{2}\,\left(0\,,\,\mp\,
\frac{\d_ta}{a^2}\right)\,.
$$

These expressions allow us to find the matrix $\Theta\,:=\,\d_tQ\,P$, and then $\Sigma^0$ and $\Sigma^1$. We start with the diagonal elements:
\begin{eqnarray*}
\theta_{11} & = & \d_t\ell_+\,\cdot\,r_+\;=\;\frac{a\,\d_ta}{1+a^2}\,-\,\frac{\d_ta}{2a} \\
& = & \frac{1}{2}\,\d_t\log\left(\frac{1+a^2}{a}\right) \\
\theta_{22} & = & \d_t\ell_-\,\cdot\,r_-\;=\;\theta_{11}\,.
\end{eqnarray*}
Thanks to these relations, the ODE \eqref{eq:ode_omega} can be explicitly solved in an exact way: for $j=1$ or $2$,
$$
\d_t\omega_j\,=\,-2\,\theta_{jj}\,=\,\d_t\log\left(\frac{a}{1+a^2}\right)\qquad\Longrightarrow\qquad
\omega_j\,=\,\log\left(\frac{a}{1+a^2}\right).
$$
Therefore, we find $\sigma_1=\sigma_2=a/(1+a^2)$ and
$$
\Sigma^0\,=\,\frac{a}{1+a^2}\,\Id\qquad\Longrightarrow\qquad
S^0\,=Q^*\,\Sigma^0\,Q\,=\,\frac{1}{2}\left(\begin{array}{cc}
a & 0 \\[1ex] 0 & 1/a(t)
\end{array}\right)\,.
$$
Notice that, thanks to the additional $L^\infty$ hypothesis, for any $p\in[1,+\infty]$ then $S^0$ is always well-defined and bounded,
and its elements have the same Zygmund regularity as $\alpha$.

Let us now construct the second part of the symmetrizer: due to the properties of $\Sigma^1$, it's enough to find $\wtilde{\sigma}_{12}$.
We use formula \eqref{eq:Sigma_1}: easy computations lead us to
$$
\ell_-\,\d_tA\,r_+\;=\;-\,\d_ta\,\xi \qquad\mbox{ and }\qquad
\ell_+\,\d_tA\,r_-\;=\;\d_ta\,\xi\,,
$$
which imply the following expression for $\wtilde{\sigma}_{12}$:
\begin{eqnarray*}
\wtilde{\sigma}_{12} & = & \frac{i\,\xi}{(\lambda_+\,-\,\lambda_-)^2}\,
\bigl(\sigma_+\,\ell_+\,\d_tA\,r_-\,-\,\sigma_2\,\ell_-\,\d_tA\,r_+\bigr) \\
& = & \frac{-i\,\d_ta}{2\,a\,(1+a^2)}\,.
\end{eqnarray*}
Finally, recalling that $\wtilde{\sigma}_{21}\,=\,\oline{\wtilde{\sigma}_{12}}$, it's immediate to get the matrix $S^1$:
$$
S^1\,=Q^*\,\Sigma^1\,Q\,=\,\frac{i\,\d_ta}{4\,a^2}\left(\begin{array}{cc}
0 & 1 \\[1ex] -1 & 0
\end{array}\right)\,.
$$

In the end, the energy associated to $u$ becomes
$$
E(t,\xi)\,=\,S\what{u}\,\cdot\,\what{u}\,=\,\frac{1}{2}\,\left(a(t)\,|\xi|^2\,|\what{u}|^2\,+\,\frac{1}{a(t)}\,|\d_t\what{u}|^2\,+\,
\frac{\d_ta}{2\,a^2}\,\Re\left(\d_t\what{u}\,\cdot\,\what{u}\right)\right)\,,
$$
which slightly differs from the one used by Tarama in \cite{Tar}.
From now on, the computations can be performed as in the general case, explained in the previous section, and they allow us to find the same results of \cite{Tar}.

We stress again the fact that the previous construction can be performed for any $p\in[1,+\infty]$.

%%%%%%%%%%%%%%%%%%%%%%%%%%%%%%%%%%%%%%%%%%%%%%%%%%%%%%%%%%%%%%%%%%%%%%%%%%%%%%%%%%%%%%%%%%%%%%%%%%%%%%%%%%
%%%%%%%%%%%%%%%%%%%%%%%%%%%%%%%%%%%%%%%%%%%%%%%%%%%%%%%%%%%%%%%%%%%%%%%%%%%%%%%%%%%%%%%%%%%%%%%%%%%%%%
\appendix
%%%%%%%%%%%%%%%%%%%%%%%%%%%%%%%%%%%%%%%%%%%%%%%%%%%%%%%%%%%%%%%%%%%%%%%%%%%%%%%%%%%%%%%%%%%%%%%%%%%%%%%%%%%%
%%%%%%%%%%%%%%%%%%%%%%%%%%%%%%%%%%%%%%%%%%%%%%%%%%%%%%%%%%%%%%%%%%%%%%%%%%%%%%%%%%%%%%%%%%%%%%%%%%%5

\section{Appendix -- Proof of Theorem \ref{t:comp}} \label{app:alg_top}

We show here the proof of Theorem \ref{t:comp}. We will follow the main steps of Proposition 4 in \cite{D2010},
based on Meyer's paralinearization method, performing suitable modifications, in order to adapt the arguments to the logarithmic instance.

First of all, we introduce the telescopic series
$$
\sum_{j=2}^{+\infty}F_j\,,\qquad\qquad\mbox{ with }\qquad F_j\,:=\,F(S_{j+1}u)\,-\,F(S_ju)\,,
$$
where we adopted the same notations of Subsection \ref{ss:L-P}.

\begin{lemma} \label{l:telescopic}
 Under the hypotheses of Theorem \ref{t:comp}, the series $\sum_jF_j$ converges to $F(u)\,-\,F(S_2u)$ in $\mc{S}'$. Moreover
one has
\begin{equation} \label{eq:telescopic}
F_j\,=\,m_j\,\Delta_ju\,,\qquad\qquad\mbox{ where }\qquad m_j\,:=\,\int^1_0F'(S_ju\,+\,\sigma\,\Delta_ju)\,d\sigma\,.
\end{equation}
\end{lemma}

\begin{proof}
 Equality \eqref{eq:telescopic} is a straightforward consequence of the mean value theorem. So, it's enough to prove the convergence
 of the series.
 
Hence, for any fixed $N\geq2$, let us estimate
$$
 \left\|F(u)\,-\,F(S_2u)\,-\,\sum_{j=2}^NF_j\right\|_{L^p}\,=\,\left\|F(u)\,-\,F(S_{N+1}u)\right\|_{L^p}\,\leq\,
\left\|u\,-\,S_{N+1}u\right\|_{L^p}\,\left\|F'\right\|_{L^\infty}\,.
$$

Let us suppose $s>0$ and $r<+\infty$ for a while. Notice that, as $\nabla u\in B^{(s-1)+\alpha\log}_{p,r}$,
\begin{equation} \label{eq:limit}
\lim_{N\rightarrow+\infty}\sum_{j\geq N}2^{j(s-1)r}\,(1+j)^{\alpha r}\,\|\Delta_j\nabla u\|^r_{L^p}\,=\,0\,.
\end{equation}

Thanks to spectral localization, we have $u-S_{N+1}u\,=\,\sum_{j\geq N+1}\Delta_ju$; then, by Bernstein inequalities we infer
\begin{eqnarray}
\left\|u\,-\,S_{N+1}u\right\|_{L^p} & \leq & \sum_{j\geq N+1}\left\|\Delta_ju\right\|_{L^p} \label{est:L^p} \\
& \leq & C\,\sum_{j\geq N+1} 2^{j(s-1)}\,(1+j)^\alpha\,\left\|\Delta_j\nabla u\right\|_{L^p}\,2^{-js}\,(1+j)^{-\alpha}\,. \nonumber
\end{eqnarray}
We now apply H\"older inequality for series, and relation \eqref{eq:limit} allows us to conclude.

If $r=+\infty$, instead, we apply Proposition \ref{p:log-emb} to reconduct ourselves to the previous case
with a different $s'>0$. If $s=0$, instead, we use the  fact that $\alpha>1$ in  \eqref{est:L^p}.
\end{proof}

\begin{rem} \label{r:telescopic}
Starting from inequality \eqref{est:L^p}, it's easy to see that the previous statement is still true even if $s=0$ and $\alpha>0$,
under the additional assumption that $\alpha r'\,>\,1$, with $r'$ defined by the relation $1/r'\,+\,1/r\,=\,1$.
\end{rem}

Let us now quote Lemma 2.63 of \cite{B-C-D}.
\begin{lemma} \label{l:m_j}
 Let $g:\,\R^2\,\longrightarrow\,\R$ smooth, and set $m_j(g)\,:=\,g(S_ju\,,\,\Delta_ju)$ for all $j\in\N$.
 
 Then, for any $u\in L^\infty$ and any $\nu\in\N^d$, there exists a positive constant $C_\nu=C_\nu(g,\|u\|_{L^\infty})$
such that, for any $j\in\N$,
$$
\left\|m_j(g)\right\|_{L^\infty}\,\leq\,C_\nu\,2^{j|\nu|}\,.
$$
\end{lemma}

Finally, we need the following lemma. Recall that, for $s\geq0$, $[s]$ denotes the biggest integer smaller than or equal to $s$.
\begin{lemma} \label{l:deriv}
 Let $s>0$, $\alpha\in\R$ and $(p,r)\in[1,+\infty]^2$.
 
 There exists a positive constant $C_{s,\alpha}$ such that, for any sequence $\left(u_j\right)_{j\in\N}$ of smooth functions which satisfy
$$
\left(\sup_{|\nu|\leq[s]+1}\left(2^{j(s-|\nu|)}\,(1+j)^{\alpha}\,\left\|\d^\nu u_j\right\|_{L^p}\right)\right)_j\,\in\,\ell^r(\N)\,,
$$
then $u\,:=\,\sum_ju_j$ belongs to $B^{s+\alpha\log}_{p,r}$ and
$$
\|u\|_{B^{s+\alpha\log}_{p,r}}\,\leq\,C_{s,\alpha}\,
\left\|\left(\sup_{|\nu|\leq[s]+1}\left(2^{j(s-|\nu|)}\,(1+j)^{\alpha}\,\left\|\d^\nu u_j\right\|_{L^p}\right)\right)_j\right\|_{\ell^r}\,.
$$

If $s=0$, $\alpha>0$ and $r=1$ one can just infer
$$
\|u\|_{B^{0+\alpha\log}_{p,\infty}}\,\leq\,C_{\alpha}\,
\left\|\left(\sup_{|\nu|\leq1}\left(2^{-j|\nu|}\,(1+j)^{\alpha}\,\left\|\d^\nu u_j\right\|_{L^p}\right)\right)_j\right\|_{\ell^1}\,.
$$
\end{lemma}

\begin{proof}
 For any $j\in\N$, we have to estimate
$$
2^{js}\,(1+j)^\alpha\,\|\Delta_ju\|_{L^p}\,\leq\,2^{js}\,(1+j)^\alpha\left(\biggl\|\sum_{k<j}\Delta_ju_k\biggr\|_{L^p}\,+\,
\biggl\|\sum_{k\geq j}\Delta_ju_k\biggr\|_{L^p}\right)\,.
$$

Let us focus on the second term: as $\|\Delta_ju_k\|_{L^p}\,\leq\,C\,\|u_k\|_{L^p}$, we have
\begin{eqnarray}
2^{js}\,(1+j)^\alpha\,\biggl\|\sum_{k\geq j}\Delta_ju_k\biggr\|_{L^p} & \leq & C\,\sum_{k\geq j}2^{s(j-k)}\,
\left(\frac{1+j}{1+k}\right)^\alpha\,(1+k)^\alpha\,2^{ks}\,\|u_k\|_{L^p} \label{est:sum_1} \\
& \leq & C\,\sum_{k\geq j}2^{-s(k-j)}\,(1+|k-j|)^{|\alpha|}\,\delta_k\,, \nonumber
\end{eqnarray}
where we have denoted $\delta_k\,:=\,\sup_{|\nu|\leq[s]+1}\left((1+k)^\alpha\,2^{k(s-|\nu|)}\,\|\d^\nu u_k\|_{L^p}\right)$.

On the other hand, for $k<j$ (and then $j\geq1$), from Bernstein inequalities we infer
$$
\|\Delta_ju_k\|_{L^p}\,\leq\,C\,2^{-j([s]+1)}\,\sup_{|\nu|=[s]+1}\left\|\d^\nu u_k\right\|_{L^p}\,.
$$
Therefore, we can write
\begin{equation} \label{est:sum_2}
2^{js}\,(1+j)^\alpha\,\biggl\|\sum_{k<j}\Delta_ju_k\biggr\|_{L^p}\,\leq\,C\,\sum_{k<j}2^{(k-j)([s]+1-s)}\,
\left(\frac{1+j}{1+k}\right)^\alpha\,\delta_k\,.
\end{equation}

Putting together estimates \eqref{est:sum_1} and \eqref{est:sum_2}, we end up with the inequality
$$
2^{js}\,(1+j)^\alpha\,\|\Delta_ju\|_{L^p}\,\leq\,C\,\left(a\,*\,\delta\right)_j\,,
$$
where we have defined $\delta\,=\,\left(\delta_k\right)_k$ and $a=\left(a_k\right)_k$, with
$$
a_k\,=\,(1+k)^{|\alpha|}\,\left(2^{-ks}\,+\,2^{-k([s]+1-s)}\right)\,.
$$
This concludes the proof of the lemma when $s>0$.

If $s=0$, it's enough to notice that, in \eqref{est:sum_1}, as $k\geq j$,
$$
\left(\frac{1+j}{1+k}\right)^\alpha\,\leq\,C\,,
$$
while the term $2^{-x}(1+x)^\alpha$, arising in \eqref{est:sum_2}, can be obviously bounded by a constant.
\end{proof}

Let us come back to the proof of Theorem \ref{t:comp}. By Lemma \ref{l:telescopic}, we have the decomposition
$$
F(u)\,-F(S_2u)\,=\,\sum_{j=2}^{+\infty}F_j\,,
$$
where $F_j$ is given by \eqref{eq:telescopic}.

\begin{rem} \label{r:u-S_0u}
Due to spectral localization, in every term $\Delta_ju$ (with $j\geq2$) of the relations in \eqref{eq:telescopic},
we can replace $u$ by $u-S_0u$, with
$$
\left\|u\,-\,S_0u\right\|_{B^{s+\alpha\log}_{p,r}}\,\leq\,C\,\left\|\nabla u\right\|_{B^{(s-1)+\alpha\log}_{p,r}}\,.
$$
\end{rem}

As a first step, we want to prove that $F(u)-F(S_2u)\,\in\,B^{s+\alpha\log}_{p,r}$:
thanks to Lemma \ref{l:deriv}, it's enough to prove that
$$
\left(\sup_{|\nu|\leq[s]+1}\left(2^{j(s-|\nu|)}\,(1+j)^\alpha\,\left\|\d^\nu F_j\right\|_{L^p}\right)\right)_j\,\in\,
\ell^r(\N)\,.
$$
If we set $g(\zeta,\omega)\,=\,\int^1_0F'(\zeta+\sigma\omega)d\sigma$, by Leibniz formula and Lemma \ref{l:m_j} we infer
\begin{eqnarray*}
\left\|\d^\nu F_j\right\|_{L^p} & \leq & \sum_{\mu\leq\nu}C_{\nu,\mu}\,2^{j|\mu|}\,C_\mu(F',J)\,2^{j(|\nu|-|\mu|)}\,\|\Delta_ju\|_{L^p} \\
& \leq & C_\nu(F',J)\,2^{j|\nu|}\,\|\Delta_ju\|_{L^p}\,.
\end{eqnarray*}
Hence, from this inequality and Remark \ref{r:u-S_0u}, we get
$$
2^{j(s-|\nu|)}\,(1+j)^\alpha\,\left\|\d^\nu F_j\right\|_{L^p}\,\leq\,C_\nu(F',J)\,c_j\,\|\nabla u\|_{B^{(s-1)+\alpha\log}_{p,r}}\,,
$$
where $\left\|c_j\right\|_{\ell^r}\,=\,1$.

So, we have proved that $F(u)-F(S_2u)\,\in\,B^{s+\alpha\log}_{p,r}$, with
$$
\left\|F(u)\,-\,F(S_2u)\right\|_{B^{s+\alpha\log}_{p,r}}\,\leq\,C\,\left\|\nabla u\right\|_{B^{(s-1)+\alpha\log}_{p,r}}\,.
$$
This implies, in particular, that its gradient is in $B^{(s-1)+\alpha\log}_{p,r}$.

Now we notice that
$$
\nabla\bigl(F(S_2u)\bigr)\,=\,F'(S_2u)\,\nabla S_2u
$$
belongs to $L^p$ with all its derivatives. In fact, this easily follows from the chain rule and Leibniz formula, keeping in mind that
$u\in L^\infty$ and $\nabla S_2u\in L^p$ with all its derivatives.

From this fact, by embeddings we gather that $\nabla\bigl(F(S_2u)\bigr)\in B^{(s-1)+\alpha\log}_{p,r}$, and then also
$\nabla(F\circ u)$ belongs to the same space.

This completes the proof of Theorem \ref{t:comp}.

%%%%%%%%%%%%%%%%%%%%%%%%%%%%%%%%%%%%%%%%%%
%%%%%%%%%%%%%%%%%%%%%%%%%%%%%%%%%%%%
\subsubsection*{Acknowledgements}
%%%%%%%%%%%%%%%%%%%%%%%%%%%%%%%%%%%%%
%%%%%%%%%%%%%%%%%%%%%%%%%%%%%%%%%%

The first three authors are members of the Gruppo Nazionale per l'Analisi Matematica, la Probabilit\`a
e le loro Applicazioni (GNAMPA) of the Istituto Nazionale di Alta Matematica (INdAM).

The third author was partially supported by 
the project ``Instabilities in Hydrodynamics'', funded by the Paris city hall (program ``\'Emergences'') and the
Fondation Sciences Math\'ematiques de Paris.
He is deeply grateful to J. Rauch for useful discussions about the problem.

{\small

%%%%%%%%%%%%%%%%%%%%%%%%%%%%%%%%%%%%%%%%%%%%%%%%%%%%%%%%%%%%%%%%
%%%%%%%%%%%%%%%%%%%%%%%%%%%%%%%%%%%%%%%%%%%%%%%%%%%%%%%%%%%%%%%%
 }
%%%%%%%%%%%%%%%%%%%%%%%%%%%%%%%%%%%%%%%%%%%%%%%%%%%%%%%%%%%%%%%%
%%%%%%%%%%%%%%%%%%%%%%%%%%%%%%%%%%%%%%%%%%%%%%%%%%%%%%%%%%%%%%%%


\begin{thebibliography}{xxx}
%%%%%%%%%%%%%%%%%%%%%%%%%%%%%%%%%%%%%%%%%%%%%%%%%%%%%%%%%%%%%%%%
%%%%%%%%%%%%%%%%%%%%%%%%%%%%%%%%%%%%%%%%%%%%%%%%%%%%%%%%%%%%%%%%


\bibitem{B-C_1994} H. Bahouri, J.-Y. Chemin:
{\it \'Equations de transport relatives \`a des champs de vecteurs non-lipschitziens et m\'ecanique des fluides},
Arch. Rational Mech. Anal., {\bf 127} (1994), n. 2, 159-181.

 \bibitem{B-C-D} H. Bahouri, J.-Y. Chemin, R. Danchin: 
``Fourier Analysis and Nonlinear Partial Differential Equations'', 
Grundlehren der Mathematischen Wissenschaften (Fundamental Principles of Mathematical Sciences),
{\bf 343}, Springer, Heidelberg (2011).

\bibitem{Bony} J.-M. Bony:
{\it Calcul symbolique et propagation des singularit\'es pour les
\'equations aux d\'eriv\'ees partielles non lin\'eaires}, Ann. Sci. \'Ecole Norm. Sup. (4), {\bf 14} (1981), 209-246.

\bibitem{Ch1995} J.-Y. Chemin:
{\it Fluides parfaits incompressibles},
Ast\'erisque, {\bf 230} (1995).

\bibitem{Cic-C} M. Cicognani, F. Colombini:
{\it Modulus of continuity of the coefficients and loss of derivatives in the strictly hyperbolic Cauchy problem},
J. Differential Equations, {\bf 221} (2006), 143-157.

\bibitem{C-DG-S} F. Colombini, E. De Giorgi, S. Spagnolo:
{\it Sur les \'equations hyperboliques avec des coefficients qui ne d\'ependent que du temps},
Ann. Scuola Normale Sup. Pisa Cl. Scienze (4), {\bf 6} (1979), no. 3, 511-559.

\bibitem{C-DS} F. Colombini, D. Del Santo:
{\it A note on hyperbolic operators with log-Zygmund coefficients},
J. Math. Sci. Univ. Tokyo {\bf 16} (2009), no. 1, 95-111.

\bibitem{C-DS-F-M_tl} F. Colombini, D. Del Santo, F. Fanelli, G. M\'etivier:
{\it Time-dependent loss of derivatives for hyperbolic operators with non-regular coefficients},
Comm. Partial Differential Equations, {\bf 38} (2013), n. 10, 1791-1817.

\bibitem{C-DS-F-M_wp} F. Colombini, D. Del Santo, F. Fanelli, G. M\'etivier:
{\it A well-posedness result for hyperbolic operators with Zygmund coefficients},
J. Math. Pures Appl. (9), {\bf 100} (2013), n. 4, 455-475. 

\bibitem{C-DS-R} F. Colombini, D. Del Santo, M. Reissig:
{\it On the optimal regularity of coefficients in hyperbolic Cauchy problems},
Bull. Sci. Math., {\bf 127} (2003), n. 4, 328-347.

\bibitem{C-M} F. Colombini, G. M\'etivier,
{\it The Cauchy problem for wave equations with non-Lipschitz coefficients; application to continuation of solutions of
some nonlinear wave equations},
Ann. Sci. \'Ecole Norm. Sup. (4) {\bf 41} (2008), n. 2, 177-220.

\bibitem{D_2005} R. Danchin:
{\it Estimates in Besov spaces for transport and transport-diffusion equations with almost Lipschitz coefficients},
Rev. Mat. Iberoamericana, {\bf 21} (2005), n. 3, 863-888.

\bibitem{D2010} R. Danchin:
{\it On the well-posedness of the incompressible density-dependent
Euler equations in the $L^p$ framework}, 
J. Differential Equations, {\bf 248} (2010), n. 8, 2130-2170.

\bibitem{D-P} R. Danchin, M. Paicu:
{\it Les th\'eor\`emes de Leray et de Fujita-Kato pour le syst\`eme de Boussinesq partiellement visqueux},
Bull. Soc. Math. France, {\bf 136} (2008), n. 2, 261-309.

\bibitem{E} L. C. Evans:
``Partial Differential Equations'',
Graduate Studies in Mathematics, {\bf 19}, American Mathematical Society, Providence, RI (1998).

\bibitem{F_phd} F. Fanelli:
``Mathematical analysis of models of non-homogeneous fluids and of hyperbolic operators with low-regularity coefficients'',
Ph.D. thesis, Scuola Internazionale Superiore di Studi Avanzati \& Universit\'e Paris-Est (2012).

\bibitem{F-Z} F. Fanelli, E. Zuazua:
{\it Weak observability estimates for $1$-D wave equations with rough coefficients},
Ann. Inst. H. Poincar\'e Anal. Non Lin\'eaire, to appear (2014).

\bibitem{H-S} A. E. Hurd, D. H. Sattinger:
{\it Questions of existence and uniqueness for hyperbolic equations with discontinuous coefficients},
Trans. Amer. Math. Soc., {\bf 132} (1968), 159-174.

\bibitem{K} T. Kato:
``Perturbation Theory for Linear Operators'',
Classics in Mathematics, Springer-Verlag, Berlin (1995).

\bibitem{M-1986} G. M\'etivier:
{\it Interactions de deux chocs pour un syst\`eme de deux lois de conservation, en dimension deux d'espace},
Trans. Amer. Math. Soc., {\bf 296} (1986), 431-479.

\bibitem{M-2008} G. M\'etivier:
``Para-differential calculus and applications to the Cauchy problem for nonlinear systems'',
Centro di Ricerca Matematica ``Ennio De Giorgi'' (CRM) Series, {\bf 5}, Edizioni della Normale, Pisa (2008).

\bibitem{M-Z} G. M\'etivier, K. Zumbrun:
{\it Large viscous boundary layers for noncharacteristic nonlinear hyperbolic problems},
Mem. Amer. Math. Soc., {\bf 175} (2005).

\bibitem{Rauch} J. Rauch:
``Hyperbolic Partial Differential Equations and Geometric Optics'',
Graduate Studies in Mathematics, {\bf 133}, American Mathematical Society, Providence, RI (2012).

\bibitem{Tar} S. Tarama:
{\it Energy estimate for wave equations with coefficients in some Besov type class},
Electron. J. Differential Equations (2007), Paper No. 85, 12 pp. (electronic).


\end{thebibliography}
\end{document}